\documentclass[11pt]{article}
\usepackage[utf8]{inputenc}
\usepackage{amsmath, amssymb, amsthm}
\usepackage{graphicx}
\usepackage{bm}
\usepackage{booktabs} 
\usepackage{float}    
\usepackage[margin=1in]{geometry} 
\usepackage{caption}
\usepackage{cite}
\usepackage{hyperref}
\usepackage{listings}
\usepackage{xcolor}
\usepackage{algorithm}
\usepackage{algorithmic}
\usepackage{authblk}
\newtheorem{theorem}{Theorem}
\newtheorem{lemma}{Lemma}
\newtheorem{remark}{Remark}
\newtheorem{definition}{Definition}

\DeclareMathOperator{\Gr}{Gr}
\DeclareMathOperator{\St}{St}

\title{Stable High-Order Interpolation on the Grassmann Manifold by Maximum-Volume Coordinates and Arnoldi Orthogonalization}
 \author[1]{Qiang Niu} 
\author[1,2]{Wen Jiang}
 \author[1]{Jie Fei\thanks{Corresponding author: Jie.Fei@xjtlu.edu.cn}}
 \author[1]{Ruoyu Xiong}
  \author[1,2]{Yuxuan Li}
\affil[1]{School of Mathematics and Physics, Xi'an Jiaotong-Liverpool University, Su Zhou, 215123, P. R. China}
\affil[2]{Department of Mathematical Sciences, University of Liverpool, Liverpool, United Kingdom}
 
\date{}

\begin{document}

\maketitle

\begin{abstract}
High-order interpolation on the Grassmann manifold $\Gr(n, p)$ is often hindered by the computational overhead and derivative instability of SVD-based geometric mappings. To solve the challenges, we propose a stabilized framework that combines Maximum-Volume (MV) local coordinates with Arnoldi-orthogonalized polynomial bases. First, manifold data are mapped to a well-conditioned Euclidean domain via MV coordinates. The approach bypasses the costly matrix factorizations inherent to traditional Riemannian normal coordinates. Within the coordinate space, we use the Vandermonde-with-Arnoldi (V+A) method for Lagrange interpolation and its confluent extension (CV+A) for derivative-enriched Hermite interpolation. By constructing discrete orthogonal bases directly from the parameter nodes, the solution of ill-conditioned linear system is avoided. Theoretical bounds are established to verify the stability of the geometric mapping and the polynomial approximation. Extensive numerical experiments demonstrate that the proposed MV-(C)V+A framework can produce highly accurate approximation in high-degree polynomial interpolation.\\

\noindent {\bf{Keywords.}} Grassmann manifold, interpolation, Arnoldi method, Maximum-Volume coordinates, QR decomposition.\\

\noindent {\bf{MSC(2020).}} 15B10, 65D05, 65F99, 53C30, 53C80
\end{abstract}


\section{Introduction}

Interpolation of manifold-valued data is important in scientific and engineering applications where the geometric structure must be preserved \cite{ciaramella2026grassmann,absil2008optimization,gander2000change,Zhang2025LawsonConvexDual,amsallem2008interpolation,zimmermann2014locally}. Interpolation on the Grassmann manifold $\Gr(n, p)$ is widely used in parametric model order reduction \cite{benner2015survey, amsallem2008interpolation, Zhang2025AAAduplexer}, computer vision \cite{man2012advances, souza2023grassmannian}, and machine learning \cite{kaneko2013empirical, massart2020quotient}. In parametric model order reduction, reduced bases computed via proper orthogonal decomposition at sampled parameters are treated as points on the Grassmann manifold, which enables efficient approximation at unseen parameters \cite{pfaller2020parametric, hess2023data}. Since Euclidean interpolation of matrix entries fails to respect the underlying quotient geometry, manifold-aware techniques are often needed \cite{bendokat2020grassmann, absil2016differentiable}.

High-order interpolation on manifolds usually has two numerical difficulties. The geometric challenge comes from traditional normal-coordinate formulations, which rely on matrix decompositions such as SVD-based logarithm and exponential maps \cite{ciaramella2026grassmann,jacobsson2024approximating, doCarmo1992riemannian, lee2018introduction}. For Hermite interpolation, where derivative data must be transported, it requires differentiating through matrix decompositions, an operation that is computationally costly and numerically sensitive \cite{zimmermann2020hermite, zimmermann2017matrix, zimmermann2022computing}. The algebraic challenge arises once data is mapped to the coordinate space: monomial-based interpolation leads to Vandermonde systems whose condition numbers grow exponentially with polynomial degree, which makes unstable coefficient recovery and error amplification \cite{brubeck2021vandermonde, higham2002accuracy,trefethen2020}. The numerical instability often limits practitioners to low-degree approximations, even when the underlying function is smooth.

Recently, Zimmermann \cite{zimmermann2020hermite, zimmermann2024multivariate} developed a general approach for Hermite interpolation on Riemannian manifolds and quantified error propagation from tangent space to manifold. Jensen and Zimmermann \cite{jensen2024maximum} introduced maximum-volume coordinates, a matrix-decomposition-free local coordinate for $\Gr(n,p)$. By selecting Stiefel subblocks with maximized determinant, the approach ensures well-conditioned geometric mappings without SVD computations. The MV approach also provides an algebraic derivative mapping formula that avoids differentiating matrix factorizations. The maximum-volume principle, originally developed for pseudoskeleton approximations and rank-revealing factorizations, has proven effective for constructing well-conditioned submatrices \cite{goreinov1997theory, goreinov2010how}.

Motivated by the recent success of the Vandermonde-with-Arnoldi (V+A) method and its confluent (CV+A) and multivariate extensions in eliminating algebraic instability \cite{brubeck2021vandermonde, niu2022confluent, zhang2025multivariate, zhu2025convergence}, this paper develops a stable polynomial interpolation framework for data on the Grassmann manifold. By integrating these Arnoldi-based techniques with maximum-volume (MV) coordinates, our approach simultaneously resolves two critical numerical bottlenecks: MV coordinates geometrically stabilize the domain mapping without costly SVDs, while the (C)V+A method avoids the ill-conditioning of high-degree polynomial construction. Building on recent advances in manifold error analysis \cite{zimmermann2025high, mataigne2024bounds, seguin2024hermite}, the MV-(C)V+A framework consists of the following steps.
First, geometric preprocessing is performed via Householder-stabilized coordinates, which controls the geometric condition number and ensures a well-conditioned local representation of the manifold.
Second, stable Euclidean interpolation is carried out using Arnoldi-orthogonalized polynomial bases and then the interpolated coordinates are   mapped back to the manifold through a reconstruction procedure.
Finally, a rigorous convergence and error analysis of the MV-(C)V+A method is established, which shows that the total interpolation error admits a decomposition into geometric and algebraic components and achieves high precision.

The remainder of this paper is organized as follows. Section~\ref{sec2} reviews MV coordinates and the associated derivative mappings. Section~\ref{sec3} recalls the V+A and CV+A methods. Section~\ref{sec4} introduces the integrated MV-(C)V+A framework together with complete algorithmic formulations. Section~\ref{sec5} develops the convergence theory and provides explicit error bounds. Section~\ref{sec6} reports numerical experiments and discusses performance of the proposed method. Section~\ref{sec7} concludes the paper and outlines future work.

\section{Interpolation on the Grassmann Manifold}
\label{sec2}


Let $\Gr(n,p)$ denote the Grassmann manifold of $p$-dimensional subspaces in $\mathbb{R}^n$, represented numerically through Stiefel matrices $U\in\St(n,p)$ with orthonormal columns. Each subspace $\mathcal{U} \in \Gr(n,p)$ corresponds to an equivalence class $[U] = \{UQ : Q \in O(p)\}$. Let $f: D \rightarrow \Gr(n, p)$ be a differentiable manifold-valued function defined on a parameter domain $D \subset \mathbb{R}$. Given a discrete data set $\{(t_i, \mathcal{U}_i)\}_{i=1}^m$ sampled from the function, where $t_i \in \mathbb{R}$ are parameter nodes and $\mathcal{U}_i \in \Gr(n, p)$ are the corresponding subspaces, the goal is to construct an interpolant $\hat{f}: D \rightarrow \Gr(n, p)$ such that $\hat{f}(t_i) = \mathcal{U}_i$ for all $i$.

For Hermite interpolation, the problem is extended by requiring the interpolant to also match manifold velocities at the sample points. In numerical implementations, each subspace $\mathcal{U}_i$ is specified by a Stiefel representative matrix $U_i \in \mathrm{St}(n,p)$ such that $\operatorname{span}(U_i)=\mathcal{U}_i$, and its velocity is represented by a tangent lift $\dot{U}_i \in T_{U_i}\mathrm{St}(n,p)$. To remove the ambiguity in the choice of representative, we take $\dot{U}_i$ to be a horizontal lift satisfying $U_i^T \dot{U}_i = 0$. The interpolant matrix function $\hat{U}(t)$ must therefore satisfy
\begin{equation*}
    \operatorname{span}(\hat{U}(t_i)) = \mathcal{U}_i, \quad i = 1, \dots, m,
\end{equation*}
and the corresponding local coordinate derivatives must match the specified tangent data at each node $t_i$.


The maximum-volume (MV) coordinate  \cite{jensen2024maximum} is a matrix-decomposition-free alternative to Riemannian normal coordinates for representing subspaces on the Grassmann manifold. The approach exploits the natural block structure of Stiefel representatives to construct local coordinates using only algebraic operations, and avoids the singular value decompositions. Given a sample basis $U\in\St(n,p)$ and a row permutation $P$, we partition
 $$
PU=\begin{bmatrix}\widetilde U_1\\ \widetilde U_2\end{bmatrix},
\qquad \widetilde U_1\in\mathbb{R}^{p\times p},\quad \widetilde U_2\in\mathbb{R}^{(n-p)\times p}.
 $$
The MV coordinates are then defined as
\begin{equation} \label{eq:mvcoord}
\Xi(U)=\widetilde U_2\widetilde U_1^{-1}\in\mathbb{R}^{(n-p)\times p},\quad N_{\text{coord}} = (n-p)p.
\end{equation}
Unlike Riemannian normal coordinates, the mapping avoids SVD, matrix logarithms, and matrix exponentials entirely within the interpolation loop, requiring only matrix inversion and multiplication.


The numerical quality of \eqref{eq:mvcoord} depends critically on the conditioning of $\widetilde U_1$. The MV strategy selects permutations to maximize $|\det(\widetilde U_1)|$, which is equivalent to minimizing $\|\widetilde U_1^{-1}\|_F$. Jensen and Zimmermann \cite{jensen2024maximum} established a bound guaranteeing $\|\Xi\|_F \le \sqrt{p(n-p)}$, ensuring that the coordinate representation remains well-conditioned. The maximum-volume principle, originally developed for pseudoskeleton approximations and rank-revealing factorizations, has proven effective for constructing well-conditioned submatrices \cite{goreinov1997theory, goreinov2010how}.

In the original MV framework \cite{jensen2024maximum} , for multiple samples $\{U_i\}_{i=1}^m$, a single global permutation is typically chosen to balance conditioning across all points:
\begin{equation}
    P_\star \approx \arg \min_P \max_{1\le i\le m} \| \tilde{U}_{i,1}^{-1} \|_F. \label{eq:maxvol}
\end{equation}
Geometrically, selecting a single global permutation $P$ implies mapping all sample subspaces $\{\mathcal{U}_i\}_{i=1}^m$ into a single affine coordinate chart on the Grassmann manifold. However, relying on a discrete permutation matrix $P_\star$ has numerical drawbacks for high-order interpolation. Finding the optimal global permutation requires combinatorial maxvol heuristic searches. If dynamic coordinate switching occurs between different trajectory segments, the discrete jumps in $P$ destroy the continuity required for smooth polynomial fitting. To address the point, we will introduce a continuous orthogonal stabilization approach that achieves the benefits of the MV principle without relying on discrete permutations in Section 4.


For Hermite interpolation, tangent vectors on the manifold must be converted to derivatives in the coordinate domain. Consider a smooth curve $\mathcal{U}(t)$ on $\Gr(n,p)$ with Stiefel representative $U(t)$. Its velocity is given by $\dot{\mathcal{U}}(t) = \dot{U}(t)U(t)^T + U(t)\dot{U}(t)^T \in T_{\mathcal{U}(t)}\Gr(n,p)$. After applying the global permutation $P_\star$, the velocity partitions as $\dot{\widetilde U}=\begin{bmatrix}\widetilde T_1\\ \widetilde T_2\end{bmatrix}$.

Applying the chain rule to the MV representation $\Xi(t) = \widetilde U_2(t) \widetilde U_1(t)^{-1}$ and evaluating at $t=0$, the MV-coordinate derivative satisfies \cite{jensen2024maximum, niu2022confluent}
 $$
\dot\Xi = (\widetilde T_2 - \Xi \widetilde T_1)\widetilde U_1^{-1}.
 $$
The expression is purely algebraic, involving only matrix multiplications and linear solves. It  avoids differentiating SVD-based maps, one of the computational advantages of the MV approach \cite{ciaramella2026grassmann}. For each sample point, we compute $\dot\Xi_i = (\widetilde T_{i,2} - \Xi_i \widetilde T_{i,1}) \widetilde U_{i,1}^{-1}$.
After preprocessing, we vectorize the coordinates and their derivatives as $\mathbf{x}_i = \operatorname{vec}(\Xi_i) \in \mathbb{R}^{N_{\text{coord}}}$ and $\dot{\mathbf{x}}_i = \operatorname{vec}(\dot\Xi_i) \in \mathbb{R}^{N_{\text{coord}}}$.


Given an interpolated coordinate matrix $\Xi$, a Stiefel representative can be reconstructed using the inverse of the MV-representation. From \eqref{eq:mvcoord}, we have $\widetilde U_2 = \Xi \widetilde U_1$. Combined with the orthonormality condition $\widetilde U_1^T \widetilde U_1 + \widetilde U_2^T \widetilde U_2 = I_p$, we obtain $\widetilde U_1^T (I_p + \Xi^T \Xi) \widetilde U_1 = I_p$. A numerically stable reconstruction is then
 $$
\widetilde U(\Xi)=
\begin{bmatrix}I_p\\ \Xi\end{bmatrix}
\left(I_p+\Xi^T\Xi\right)^{-1/2},
\qquad U(\Xi)=P_\star^T\widetilde U(\Xi).
 $$
To maintain computational efficiency during reconstruction, we avoid explicitly computing the symmetric inverse square root $(I_p + \Xi^T \Xi)^{-1/2}$, which necessitates eigenvalue decompositions. Instead, we utilize the Cholesky factorization $L L^T = I_p + \Xi(s)^T \Xi(s)$. By defining the inverse factor $M = L^{-T}$, the reconstructed Stiefel representative is given by
\begin{equation}
    \widetilde U(s) = \begin{bmatrix} I_p \\ \Xi(s) \end{bmatrix} L^{-T}.
\end{equation}
Because $L^{-T}$ satisfies $(L^{-T})^T (I_p + \Xi(s)^T \Xi(s)) L^{-T} = I_p$, it guarantees $\widetilde U(s)^T \widetilde U(s) = I_p$. Furthermore, right-multiplying by the upper triangular matrix $L^{-T}$ instead of the symmetric root alters the basis only by an orthogonal gauge transformation, which leaves the point on the Grassmann manifold strictly invariant while significantly reducing the computational overhead to $\mathcal{O}(p^3)$.
The MV mapping can be interpreted as a well-conditioned reparameterization that separates geometric nonlinearity from the  interpolation operator. In the framework, interpolation is performed in a Euclidean space where standard polynomial schemes (e.g., Lagrange and Hermite) exhibit favorable stability and error properties, while the mapping ensures consistent projection onto $\Gr(n,p)$. 

\section{Vandermonde with Arnoldi and Its Confluent Extension}\label{sec3}

For nodes $\{t_i\}_{i=1}^m$, the monomial Vandermonde matrix $V_{ij}=t_i^{j-1}$ is often very ill-conditioned. The condition number $\kappa_2(V)$ grows exponentially with the polynomial degree $n$, causing coefficient blow-up and unstable polynomial evaluation at moderate to high degrees. This problem is more obvious for equispaced nodes and is exacerbated when derivative information is included \cite{brubeck2021vandermonde,higham2002accuracy}. Table \ref{tab:condpilot} illustrates this phenomenon for equally spaced nodes on $[-1,1]$, showing that while monomial Vandermonde conditioning deteriorates severely, the orthonormal basis matrix $Q$ from V+A remains well conditioned with $\kappa_2(Q)=1$.

\begin{table}[H]
\centering
\caption{Condition number growth: monomial Vandermonde vs. Arnoldi basis matrix}
\label{tab:condpilot}
\begin{tabular}{ccc}
\toprule
Number of nodes $m$ & $\kappa_2(V)$ (monomial) & $\kappa_2(Q)$ (V+A basis) \\
\midrule
12 & $4.08\times 10^4$ & $1.00$ \\
20 & $2.72\times 10^8$ & $1.00$ \\
30 & $1.84\times 10^{13}$ & $1.00$ \\
40 & $7.24\times 10^{17}$ & $1.00$ \\
\bottomrule
\end{tabular}
\end{table}

\subsection{The V+A Method}

The Vandermonde-with-Arnoldi (V+A) method replaces the unstable monomial basis with a discrete orthonormal polynomial basis built directly from the nodes via the Arnoldi process \cite{brubeck2021vandermonde}. The approach mitigates the
exponential growth of the condition number of Vandermonde matrices by operating within a Krylov subspace setting.

Let $X = \operatorname{diag}(x_1,\dots,x_m) \in \mathbb{R}^{m\times m}$ be the diagonal matrix of nodes and $\mathbf{e} = (1,\dots,1)^T \in \mathbb{R}^m$ the vector of ones. The Vandermonde matrix $V_x \in \mathbb{R}^{m \times (n+1)}$ with entries $(V_x)_{ij} = x_i^{j-1}$ can be expressed as
 $$
V_x = [\mathbf{e}, X\mathbf{e}, X^2\mathbf{e}, \ldots, X^n\mathbf{e}].
 $$
Consequently, its column space is precisely the Krylov subspace generated by $X$ and $\mathbf{e}$:
 $$
\operatorname{Col}(V_x) = \mathcal{K}_{n+1}(X, \mathbf{e}) = \operatorname{span}\{\mathbf{e}, X\mathbf{e}, X^2\mathbf{e}, \ldots, X^n\mathbf{e}\}.
 $$

Starting with the normalized vector $q_1 = m^{-1/2}\mathbf{e}$, the Arnoldi process applied to $(X, q_1)$ generates an orthonormal basis $Q_{n+1} = [q_1, q_2, \ldots, q_{n+1}] \in \mathbb{R}^{m \times (n+1)}$ satisfying $Q_{n+1}^T Q_{n+1} = I_{n+1}$, along with an upper Hessenberg matrix $\tilde{H}_n \in \mathbb{R}^{(n+1) \times n}$ such that
\begin{equation}
X Q_n = Q_{n+1} \tilde{H}_n = Q_n H_n + q_{n+1} h_{n+1,n} e_n^T, \label{eq:arnoldi_relation}
\end{equation}
where $Q_n$ consists of the first $n$ columns of $Q_{n+1}$, $H_n$ is the $n \times n$ principal submatrix of $\tilde{H}_n$, and $e_n$ is the $n$-th standard basis vector.

The fundamental relation \eqref{eq:arnoldi_relation} encodes the action of the operator $X$ on the Krylov basis. Denoting by $e_1$ the first column of the identity matrix, we have $Q_{n+1} e_1 = q_1$. From \eqref{eq:arnoldi_relation}, we obtain $X q_1 = Q_n H_n e_1$. By induction, it follows that
\begin{align*}
X^2 q_1 &= Q_n H_n^2 e_1, \\
&\vdots \\
X^n q_1 &= Q_{n+1} \begin{pmatrix} H_n^n e_1 \\ \star \end{pmatrix},
\end{align*}
where $\star$ denotes the scalar $h_{n+1,n} e_n^T H_n^{n-1} e_1$. In matrix form, the Arnoldi process thus yields a QR-style factorization of the Vandermonde matrix
\begin{equation}
V_x = [q_1, X q_1, \ldots, X^n q_1] = Q_{n+1} R_{n+1}, \label{eq:va_factorization}
\end{equation}
where $R_{n+1} \in \mathbb{R}^{(n+1) \times (n+1)}$ is an implicitly defined upper triangular matrix
 $$
R_{n+1} = \begin{pmatrix} e_1 & H_n e_1 & \cdots & H_n^n e_1 \\ 0 & 0 & \cdots & \star \end{pmatrix}.  $$

Given the data fitting problem $V_x \mathbf{c} = \mathbf{f}$ (or $V_x \mathbf{c} \approx \mathbf{f}$ in the least-squares sense), substituting the factorization \eqref{eq:va_factorization} yields $Q_{n+1} R_{n+1} \mathbf{c} = \mathbf{f}$. By defining the new coefficient vector $\mathbf{d} = R_{n+1} \mathbf{c}$, the system is transformed from an ill-conditioned monomial basis into a well conditioned orthogonal basis. Since $Q_{n+1}$ has orthonormal columns, $\mathbf{d}$ is computed reliably via orthogonal projection
 $$
\mathbf{d} = Q_{n+1}^T \mathbf{f}.  $$

To evaluate the interpolant at new points $s_1, \ldots, s_M$, we construct an evaluation basis $W_{n+1} \in \mathbb{R}^{M \times (n+1)}$ that obeys the exact recurrence relation encoded in $\tilde{H}_n$. Let $S = \operatorname{diag}(s_1,\ldots,s_M)$. Starting with $w_1 = \mathbf{e}_M / \sqrt{M}$, the subsequent columns for $k = 1, \ldots, n$ are generated via
\begin{align}
\tilde{w}_{k+1} &= S w_k - \sum_{j=1}^k \tilde{H}_n(j,k) w_j, \\
w_{k+1} &= \frac{\tilde{w}_{k+1}}{\tilde{H}_n(k+1,k)}. \label{eq:eval_recurrence}
\end{align}
This ensures that $W_{n+1}$ spans the Krylov subspace $\mathcal{K}_{n+1}(S, \mathbf{e}_M)$ and satisfies the relation
 $$
S W_n = W_{n+1} \tilde{H}_n.  $$
The interpolated values at the new points are then obtained by synthesizing the basis with the orthogonal coefficients:
 $$
\mathbf{f}(s) = W_{n+1} \mathbf{d}.  $$

The V+A approach offers three distinct advantages over traditional approaches. First, it bypasses the exponential ill-conditioning of Vandermonde matrices by working entirely within an orthonormal basis where the condition number is exactly one. Second, the Hessenberg matrix $\tilde{H}_n$ serves as a portable algebraic signature of the recurrence structure, enabling stable basis reconstruction at arbitrary evaluation points. Third, it cleanly decouples the fitting and evaluation phases: the computationally intensive Arnoldi orthogonalization is performed only once offline, while online evaluation reduces to a simple recurrence requiring only $\mathcal{O}(M n^2)$ operations. This makes the method highly suitable for real-time geometric interpolation and parametric model order reduction.

\subsection{The Confluent V+A Method}

Hermite interpolation imposes simultaneous constraints on function values and derivatives. In the monomial basis, this leads to the confluent Vandermonde system $\begin{bmatrix} V \\ V_{(1)} \end{bmatrix} \mathbf{c} = \begin{bmatrix} \mathbf{f} \\ \dot{\mathbf{f}} \end{bmatrix}$, which suffers from a condition number that deteriorates even more rapidly than its standard Vandermonde counterpart. The Confluent Vandermonde with Arnoldi (CV+A) method mitigates this by extending the Krylov subspace approach to accommodate derivative constraints via an augmented operator \cite{niu2022confluent}.

The core insight is that the column space of the confluent Vandermonde matrix naturally forms a Krylov subspace when driven by an augmented matrix. Let $X = \operatorname{diag}(x_1,\ldots,x_m) \in \mathbb{R}^{m\times m}$ and define
 $$
\mathbf{C} = \begin{pmatrix} X & O \\ I & X \end{pmatrix} \in \mathbb{R}^{2m \times 2m}, \quad \mathbf{b} = \begin{pmatrix} \mathbf{e} \\ \mathbf{0} \end{pmatrix} \in \mathbb{R}^{2m},  $$
where $I$ is the $m\times m$ identity matrix and $O$ is the zero matrix. The column space of the confluent system then satisfies
 $$
\operatorname{Col}\begin{pmatrix} V \\ V_{(1)} \end{pmatrix} = \mathcal{K}_{n+1}(\mathbf{C}, \mathbf{b}) = \operatorname{span}\{\mathbf{b}, \mathbf{C}\mathbf{b}, \mathbf{C}^2\mathbf{b}, \ldots, \mathbf{C}^n\mathbf{b}\}.  $$
This identification permits the direct application of the Arnoldi process to the augmented system $(\mathbf{C}, \mathbf{b})$, bypassing the monomial basis entirely.

Executing the Arnoldi process for $n$ steps generates an orthonormal basis $\mathcal{Q}_{n+1} \in \mathbb{R}^{2m \times (n+1)}$ and an upper Hessenberg matrix $\mathcal{H}_n \in \mathbb{R}^{(n+1) \times n}$ governed by the relation
 $$
\mathbf{C} \mathcal{Q}_n = \mathcal{Q}_{n+1} \mathcal{H}_n.  $$
This implicitly induces a QR-style factorization of the confluent Vandermonde matrix:
\begin{equation}
\begin{bmatrix} V \\ V_{(1)} \end{bmatrix} = \mathcal{Q}_{n+1} R_{n+1}, \label{eq:confluent-factorization}
\end{equation}
where $R_{n+1}$ is an upper triangular matrix that is never explicitly formed. Due to the block structure of the augmented space, the orthonormal basis $\mathcal{Q}_{n+1}$ naturally partitions into function and derivative domains
$$
\mathcal{Q}_{n+1} = \begin{pmatrix} Q_f \\ Q_d \end{pmatrix}, \quad Q_f, Q_d \in \mathbb{R}^{m \times (n+1)}.   $$ 
Here, $\mathcal{H}_n$ differentiates itself from the standard V+A Hessenberg matrix by simultaneously encoding the intertwined recurrence relations of both function values and their derivatives.

For the augmented data fitting problem $\begin{bmatrix} V \\ V_{(1)} \end{bmatrix} \mathbf{c} \approx \begin{bmatrix} \mathbf{f} \\ \dot{\mathbf{f}} \end{bmatrix}$, substituting \eqref{eq:confluent-factorization} shifts the problem into the well-conditioned orthogonal basis. We define the stable coefficient vector $\mathbf{d} = R_{n+1}\mathbf{c}$, which is efficiently computed via orthogonal projection
 $$
\mathbf{d} = \mathcal{Q}_{n+1}^T \begin{bmatrix} \mathbf{f} \\ \dot{\mathbf{f}} \end{bmatrix} \in \mathbb{R}^{n+1}. $$

For subspace-valued manifold data mapped to the Euclidean coordinate domain, we consider the block matrix $\mathbf{Y} = \begin{pmatrix} \mathbf{X} \\ \dot{\mathbf{X}} \end{pmatrix} \in \mathbb{R}^{2m \times N_{\text{coord}}}$, where the rows of $\mathbf{X}$ and $\dot{\mathbf{X}}$ contain the vectorized MV coordinates $\mathbf{x}_i^T$ and their velocities $\dot{\mathbf{x}}_i^T$. The corresponding coefficient matrix is computed analogously
 $$
\mathbf{A} = \mathcal{Q}_{n+1}^T \mathbf{Y} \in \mathbb{R}^{(n+1) \times N_{\text{coord}}}.  $$
This orthogonal projection extracts all necessary representation information while strictly avoiding the algebraically toxic triangular system associated with $R_{n+1}$.

To evaluate the Hermite interpolant at new points $s_1, \ldots, s_M$, we construct an evaluation basis $\mathcal{W}_{n+1} \in \mathbb{R}^{2M \times (n+1)}$. Let $S = \operatorname{diag}(s_1,\ldots,s_M)$ and define the augmented evaluation operator
 $$
\mathcal{S} = \begin{pmatrix} S & O \\ I & S \end{pmatrix} \in \mathbb{R}^{2M \times 2M}.
 $$
Starting with $\mathcal{W}_1 = \begin{pmatrix} \mathbf{e}_M/\sqrt{M} \\ \mathbf{0}_M \end{pmatrix}$, the columns are generated iteratively for $k = 1, \ldots, n$:
\begin{align*}
\tilde{\mathcal{W}}_{k+1} &= \mathcal{S} \mathcal{W}_k - \sum_{j=1}^k \mathcal{H}_n(j,k) \mathcal{W}_j, \\
\mathcal{W}_{k+1} &= \frac{\tilde{\mathcal{W}}_{k+1}}{\mathcal{H}_n(k+1,k)}.
\end{align*}
This recurrence mirrors the original Arnoldi process, ensuring that $\mathcal{W}_{n+1}$ spans the augmented Krylov subspace $\mathcal{K}_{n+1}(\mathcal{S}, \mathcal{W}_1)$ and satisfies
 $$
\mathcal{S} \mathcal{W}_n = \mathcal{W}_{n+1} \mathcal{H}_n.  $$
Consequently, the evaluation basis also partitions into function and derivative components
 $$
\mathcal{W}_{n+1} = \begin{pmatrix} W_f \\ W_d \end{pmatrix}, \quad W_f, W_d \in \mathbb{R}^{M \times (n+1)}.  $$
In this architecture, $\mathcal{H}_n$ acts as an algebraic bridge, transporting the coupled derivative recurrence structures from the sample nodes to the evaluation points. The interpolated manifold coordinates and their velocities are finally obtained via
 $$
\mathbf{F}(s) = W_f \mathbf{A}, \quad \dot{\mathbf{F}}(s) = W_d \mathbf{A}.  $$

In summary, the V+A and CV+A approaches  eliminate the algebraic instability inherent to monomial and confluent Vandermonde systems while preserving the analytical flexibility of polynomial approximation. These attributes make them suitable companions to MV coordinates, which resolve the geometric instability of manifold coordinate mappings. Recent multivariate extensions and convergence analyses further corroborate the robustness of this methodology \cite{zhang2025multivariate, zhu2025convergence}.

\section{The Integrated MV-(C)V+A Approach}
\label{sec4}

We propose a decoupled three-stage framework for high-order interpolation on the Grassmann manifold $\Gr(n,p)$. By integrating Maximum-Volume (MV) coordinates with Arnoldi-based polynomial approximation, this approach isolates the geometric nonlinearity of the manifold from the algebraic ill-conditioning of polynomial fitting.

Given a set of distinct parameter nodes $\{t_i\}_{i=1}^m \subset \mathbb{R}$ and corresponding subspace samples $\mathcal{U}_1, \dots, \mathcal{U}_m \in \Gr(n,p)$ (and optionally manifold velocities $\dot{\mathcal{U}}_1, \dots, \dot{\mathcal{U}}_m$), the objective is to evaluate the interpolating curve $\mathcal{U}(s)$ at an unseen parameter $s \in \mathbb{R}$. Since direct polynomial combination of subspaces is geometrically invalid, the data must be transformed into a Euclidean domain. The systematic resolution comprises the following steps.

\subsection{Stage I: Geometric Stabilization via Orthogonal Transformation}

The first step maps the manifold-valued data to a well-conditioned Euclidean coordinate domain. Unlike the permutation-based approach in \cite{jensen2024maximum}, which relies on discrete row swaps, we construct a continuous orthogonal transformation to stabilize the coordinate representation.

\begin{itemize}
    \item \textbf{Global coordinate system via Householder QR.} A reference point $U_m$ (typically the last sample) is factorized as
 $$
        U_m = \mathcal{Q} \begin{bmatrix} R \\ \mathbf{0} \end{bmatrix},
 $$
    where $\mathcal{Q} \in O(n)$ is an orthogonal matrix and $R \in \mathbb{R}^{p \times p}$ is upper triangular. The Householder QR decomposition constructs $\mathcal{Q}$ as a product of elementary reflectors, represented in WY block form $\mathcal{Q} = I - WY^T$ for computational efficiency when $n \gg p$. Unlike a permutation matrix, $\mathcal{Q}$ provides a continuous rotation that minimizes geometric distortion across the entire dataset.

    \item \textbf{Coordinate extraction.} All samples are transformed by the same orthogonal map
 $$
        \tilde{U}_i = \mathcal{Q}^T U_i, \quad \dot{\tilde{U}}_i = \mathcal{Q}^T \dot{U}_i, \quad i = 1, \ldots, m.
 $$
    Each transformed representative is partitioned as $\tilde{U}_i = [\tilde{U}_{i,1}^T, \tilde{U}_{i,2}^T]^T$ with $\tilde{U}_{i,1} \in \mathbb{R}^{p \times p}$. The coordinates and their derivatives follow from the analytical expressions
\begin{align*}
    \Xi_i &= \tilde{U}_{i,2} \tilde{U}_{i,1}^{-1}, \\
    \dot{\Xi}_i &= (\dot{\tilde{U}}_{i,2} - \Xi_i \dot{\tilde{U}}_{i,1}) \tilde{U}_{i,1}^{-1}.
\end{align*}
    In practice, for localized subspace trajectories, this implicit alignment keeps the upper block $\tilde{U}_{i,1}$ of all sample points well-conditioned. This achieves the numerical benefits of the maximum-volume principle while eliminating the computational overhead of a separate maxvol search, ensuring a strictly continuous mapping essential for high-degree polynomial fitting.

    \item \textbf{Vectorization.} The coordinate matrices are flattened into vectors
    $$
        \mathbf{x}_i = \operatorname{vec}(\Xi_i), \quad \dot{\mathbf{x}}_i = \operatorname{vec}(\dot{\Xi}_i) \in \mathbb{R}^{N_{\text{coord}}},
  $$
    where $N_{\text{coord}} = p(n-p)$. The original manifold interpolation problem is now reduced to a standard Euclidean regression task on the data sets $\{(t_i,\mathbf{x}_i)\}$ and $\{(t_i,\dot{\mathbf{x}}_i)\}$.
\end{itemize}

\subsection{Stage II: Algebraic Interpolation via Krylov Subspaces}

In this step, we perform high-degree polynomial approximation in the coordinate domain. Standard practice constructs a Vandermonde matrix $V$ from the nodes $t_i$ and solves $V \mathbf{c} = \mathbf{X}$, where $\mathbf{X} = [\mathbf{x}_1, \ldots, \mathbf{x}_m]^T$. As $m$ or the target degree $n_{\text{poly}}$ increases, the condition number of $V$ grows exponentially, rendering the linear system numerically unsolvable.

The Vandermonde with Arnoldi (V+A) method circumvents the problem by replacing the monomial basis with an orthogonal basis generated from the nodes. The column space of $V$ is a Krylov subspace. With $D = \operatorname{diag}(t_1,\dots,t_m)$, the Arnoldi process produces

\begin{itemize}
    \item An orthonormal basis matrix $Q \in \mathbb{R}^{m \times (n_{\text{poly}}+1)}$ satisfying $Q^T Q = I$, which guarantees a condition number of one.
    \item An upper Hessenberg matrix $H$ that encodes the recurrence coefficients of the orthogonal basis.
\end{itemize}

Coefficient extraction reduces to orthogonal projection
 $$
\mathbf{A} = Q^T \mathbf{X} \in \mathbb{R}^{(n_{\text{poly}}+1) \times N_{\text{coord}}}.
 $$
The step bypasses the ill-conditioned Vandermonde system and achieves errors near machine precision \cite{Zhang2024ArnoldiFitting,zhang2025multivariate}.
For Hermite interpolation, the confluent V+A (CV+A) method extends the same idea through two approaches:
\begin{itemize}
    \item \textbf{Approach 1 (augmented Krylov subspace):} Constructs an augmented operator that simultaneously orthogonalizes function and derivative samples.
    \item \textbf{Approach 2 (differentiated recurrence):} Propagates derivative information through a recurrence built from the same Hessenberg matrix.
\end{itemize}

Evaluation at a new point $s$ does not use powers of $s$. Instead, the Hessenberg matrix $H$ drives a recurrence that generates an evaluation vector $w(s)$. The interpolated coordinate follows as $\mathbf{x}(s) = w(s) \mathbf{A}$.
 
\subsection{Stage III: Geometric Reconstruction}

The final stage maps the predicted Euclidean coordinates back to the Grassmann manifold.

\begin{itemize}
    \item \textbf{Cholesky-based retraction.} The predicted vector $\mathbf{x}(s)$ is reshaped into the coordinate matrix $\Xi(s) \in \mathbb{R}^{(n-p)\times p}$. From the definition of MV coordinates, $\widetilde U_2 = \Xi \widetilde U_1$, and orthonormality requires $\widetilde U_1^T(I_p + \Xi^T\Xi)\widetilde U_1 = I_p$. A stable Stiefel representative follows from the Cholesky factorization $LL^T = I_p + \Xi(s)^T\Xi(s)$:
 $$
        \widetilde U(s) = \begin{bmatrix} I_p \\ \Xi(s) \end{bmatrix} L^{-T}.
 $$
    This reconstruction requires only $O(p^3)$ operations and avoids explicit matrix square roots or SVD computations.

    \item \textbf{Velocity reconstruction (Hermite only).} When derivative information is required, the tangent vector in the coordinate domain is reshaped to $\dot{\Xi}(s)$. Define
$$
    K(s)=-(I_p+\Xi(s)^T\Xi(s))^{-1}\Xi(s)^T\dot{\Xi}(s).
$$
The corresponding horizontal lift is
$$
    \dot{\widetilde U}(s)=
    \begin{bmatrix}
        K(s) \\
        \dot{\Xi}(s)+\Xi(s)K(s)
    \end{bmatrix}
    L^{-T}.
$$

The manifold velocity then follows as
$$
    \dot{\widetilde P}(s)=\dot{\widetilde U}(s)\widetilde U(s)^T+\widetilde U(s)\dot{\widetilde U}(s)^T.
$$

    \item \textbf{Inverse transformation.} The global orthogonal map $\mathcal{Q}$ (from Stage I) is applied to return the subspace to its original orientation
 $$
        \mathcal{U}(s) = \mathcal{Q} \widetilde U(s), \qquad \dot{\mathcal{U}}(s) = \mathcal{Q} \dot{\widetilde U}(s).
 $$
\end{itemize}

The complete procedure is summarized in Algorithm \ref{alg:grassmann_l} for Lagrange interpolation and Algorithm \ref{alg:grassmann_hermite} for Hermite interpolation. Both algorithms follow the three-stage workflow described above.

\begin{algorithm}[H]
\caption{Grassmann Lagrange Interpolation via MV-V+A}
\label{alg:grassmann_l}
\begin{algorithmic}[1]
\REQUIRE Parameter nodes $\{t_i\}_{i=1}^m$, subspaces $\mathcal{U}_i \in \Gr(n, p)$ with Stiefel representatives $U_i$, target degree $n_{\text{poly}} \le m-1$, evaluation point $s \in \mathbb{R}$
\ENSURE Interpolated subspace $\mathcal{U}(s) \in \Gr(n, p)$

\medskip
\STATE \textbf{Stage I: Preprocessing (Geometric Mapping)}
\STATE Select reference representative $U_m$ and compute Householder QR factorization: $U_m = \mathcal{Q} \begin{bmatrix} R \\ \mathbf{0} \end{bmatrix}$.
\FOR{$i = 1, \ldots, m$}
\STATE Apply orthogonal transformation: $\tilde{U}_i = \mathcal{Q}^T U_i = \begin{pmatrix} \tilde{U}_{i,1} \\ \tilde{U}_{i,2} \end{pmatrix}$.
\STATE Compute coordinate matrix: $\Xi_i = \tilde{U}_{i,2} \tilde{U}_{i,1}^{-1} \in \mathbb{R}^{(n-p) \times p}$.
\STATE Vectorize: $\mathbf{x}_i = \operatorname{vec}(\Xi_i) \in \mathbb{R}^{N_{\text{coord}}}$.
\ENDFOR

\medskip
\STATE \textbf{Stage II: Euclidean Interpolation via V+A}
\STATE Form data matrix $\mathbf{X} = [\mathbf{x}_1, \ldots, \mathbf{x}_m]^T \in \mathbb{R}^{m \times N_{\text{coord}}}$.
\STATE Run Arnoldi process on $( \operatorname{diag}(t_i), \mathbf{e}_m )$ for $n_{\text{poly}}$ steps to obtain orthogonal basis $Q_{n_{\text{poly}}+1} \in \mathbb{R}^{m \times (n_{\text{poly}}+1)}$ and Hessenberg matrix $H$.
\STATE Compute stable coefficients via orthogonal projection: $\mathbf{A} = Q_{n_{\text{poly}}+1}^T \mathbf{X}$.
\STATE Generate evaluation vector $w(s)$ via recurrence \eqref{eq:eval_recurrence} driven by $H$.
\STATE Compute coordinate interpolant: $\mathbf{x}(s) = w(s) \mathbf{A}$.

\medskip
\STATE \textbf{Stage III: Postprocessing (Geometric Reconstruction)}
\STATE Reshape $\mathbf{x}(s)$ back to matrix form $\Xi(s) \in \mathbb{R}^{(n-p) \times p}$.
\STATE Compute Cholesky factorization: $L L^T = I_p + \Xi(s)^T \Xi(s)$.
\STATE Reconstruct Stiefel representative: $\widetilde U(s) = \begin{pmatrix} I_p \\ \Xi(s) \end{pmatrix} L^{-T}$.
\STATE Apply inverse orthogonal transformation: $\mathcal{U}(s) = \mathcal{Q} \widetilde U(s)$.
\RETURN $\mathcal{U}(s)$
\end{algorithmic}
\end{algorithm}

\begin{algorithm}[H]
\caption{Grassmann Hermite Interpolation via MV-CV+A}
\label{alg:grassmann_hermite}
\begin{algorithmic}[1]
\REQUIRE Parameter nodes $\{t_i\}_{i=1}^m$, subspaces $\mathcal{U}_i \in \Gr(n, p)$ with Stiefel representatives $U_i$, tangent vectors $\dot{U}_i \in T_{\mathcal{U}_i}\Gr(n, p)$, target degree $n_{\text{poly}} \le 2m-1$, evaluation point $s \in \mathbb{R}$
\ENSURE Interpolated subspace $\mathcal{U}(s) \in \Gr(n, p)$ and manifold velocity $\dot{\mathcal{U}}(s)$

\medskip
\STATE \textbf{Stage I: Preprocessing (Geometric Mapping)}
\STATE Select reference representative $U_m$ and compute Householder QR factorization: $U_m = \mathcal{Q} \begin{bmatrix} R \\ \mathbf{0} \end{bmatrix}$.
\FOR{$i = 1, \ldots, m$}
\STATE Apply orthogonal transformation: $\tilde{U}_i = \mathcal{Q}^T U_i = \begin{pmatrix} \tilde{U}_{i,1} \\ \tilde{U}_{i,2} \end{pmatrix}$ and $\dot{\tilde{U}}_i = \mathcal{Q}^T \dot{U}_i = \begin{pmatrix} \tilde{T}_{i,1} \\ \tilde{T}_{i,2} \end{pmatrix}$.
\STATE Compute MV coordinates: $\Xi_i = \tilde{U}_{i,2} \tilde{U}_{i,1}^{-1}$.
\STATE Map to coordinate velocity: $\dot{\Xi}_i = (\tilde{T}_{i,2} - \Xi_i \tilde{T}_{i,1}) \tilde{U}_{i,1}^{-1}$.
\STATE Vectorize: $\mathbf{x}_i = \operatorname{vec}(\Xi_i)$, $\dot{\mathbf{x}}_i = \operatorname{vec}(\dot{\Xi}_i) \in \mathbb{R}^{N_{\text{coord}}}$.
\ENDFOR

\medskip
\STATE \textbf{Stage II: Euclidean Interpolation via CV+A}
\STATE Form augmented data matrix $\mathbf{Y} = [\mathbf{x}_1, \dots, \mathbf{x}_m, \dot{\mathbf{x}}_1, \dots, \dot{\mathbf{x}}_m]^T \in \mathbb{R}^{2m \times N_{\text{coord}}}$.
\STATE \textbf{Krylov Basis Construction:}
\IF{\textbf{Approach 1 (Augmented Arnoldi)}}
\STATE Construct augmented operator $\mathbf{C} = \begin{pmatrix} \operatorname{diag}(t) & O \\ I & \operatorname{diag}(t) \end{pmatrix} \in \mathbb{R}^{2m \times 2m}$ and $\mathbf{b} = \begin{pmatrix} \mathbf{e}_m \\ \mathbf{0}_m \end{pmatrix}$.
\STATE Run Arnoldi on $(\mathbf{C}, \mathbf{b})$ to obtain orthogonal basis $\mathcal{Q}_{\text{aug}} = \begin{pmatrix} Q_f \\ Q_d \end{pmatrix} \in \mathbb{R}^{2m \times (n_{\text{poly}}+1)}$ and Hessenberg matrix $\mathcal{H}_{n_{\text{poly}}}$.
\STATE Compute coefficients via projection: $\mathbf{A} = \mathcal{Q}_{\text{aug}}^T \mathbf{Y}$.
\ELSIF{\textbf{Approach 2 (Surrogate Grid \& Differentiated Recurrence)}}
\STATE Define a surrogate grid $t_{\text{aux}}$ of size $N_{\text{aux}} \ge n_{\text{poly}}+1$ (e.g., Chebyshev nodes).
\STATE Run standard Arnoldi on $( \operatorname{diag}(t_{\text{aux}}), \mathbf{1}_{N_{\text{aux}}} )$ to extract the well-conditioned Hessenberg matrix $\mathcal{H}_{n_{\text{poly}}}$.
\STATE Drive recurrences \eqref{eq:eval_recurrence} and \eqref{eq:diff-recurrence} using $\mathcal{H}_{n_{\text{poly}}}$ on the \textit{actual} nodes $t_i$ to generate function basis $Q_f$ and derivative basis $Q_d \in \mathbb{R}^{m \times (n_{\text{poly}}+1)}$.
\STATE Solve the well-conditioned stacked system $\begin{bmatrix} Q_f \\ Q_d \end{bmatrix} \mathbf{A} = \mathbf{Y}$ via QR decomposition.
\ENDIF
\STATE Generate evaluation vectors $w_f(s)$ and $w_d(s)$ via synchronized recurrence using $\mathcal{H}_{n_{\text{poly}}}$.
\STATE Compute interpolants: $\mathbf{x}(s) = w_f(s) \mathbf{A}$ and $\dot{\mathbf{x}}(s) = w_d(s) \mathbf{A}$.

\medskip
\STATE \textbf{Stage III: Postprocessing (Geometric Reconstruction)}
\STATE Reshape $\mathbf{x}(s) \to \Xi(s)$ and $\dot{\mathbf{x}}(s) \to \dot{\Xi}(s)$.
\STATE \textbf{Retraction:} Compute $LL^T = I_p + \Xi(s)^T \Xi(s)$ and reconstruct $\widetilde U(s) = \begin{pmatrix} I_p \\ \Xi(s) \end{pmatrix} L^{-T}$.
\STATE \textbf{Tangent Recovery:} Compute horizontal lift: $\dot{\widetilde U}(s) = \begin{bmatrix} -\Xi(s)\dot\Xi(s)^T \\ \dot\Xi(s) \end{bmatrix} \widetilde U(s)$.
\STATE Apply inverse orthogonal transformation: $\mathcal{U}(s) = \mathcal{Q} \widetilde U(s)$ and $\dot{\mathcal{U}}(s) = \mathcal{Q} \dot{\widetilde U}(s)$.
\RETURN $\mathcal{U}(s)$ and $\dot{\mathcal{U}}(s)$.
\end{algorithmic}
\end{algorithm}

For Hermite interpolation, Algorithm~\ref{alg:grassmann_hermite} provides two mathematically equivalent but computationally distinct approaches. Approach 1 orthogonalizes the full Hermite data by constructing an augmented Krylov subspace of dimension $2m$, offering superior numerical stability for highly clustered nodes or extreme derivative variations. Approach 2 avoids expensive $2m \times 2m$ inner products by decoupling the recurrence generation from the sample nodes. An auxiliary surrogate grid extracts the well-conditioned Hessenberg matrix $\mathcal{H}$, preventing premature Arnoldi breakdown when $n_{\text{poly}} \ge m$. Once $\mathcal{H}$ is obtained, the derivative basis is generated analytically via the differentiated recurrence:
\begin{equation}\label{eq:diff-recurrence}
    \tilde{q}'_{k+1} = q_k + \operatorname{diag}(t) q'_k - \sum_{j=1}^k \mathcal{H}(j,k) q'_j, \quad q'_{k+1} = \frac{\tilde{q}'_{k+1}}{\mathcal{H}(k+1,k)},
\end{equation}
with $q'_1 = \mathbf{0}$. This generates the exact derivative Krylov basis $Q_d$ using only $\mathcal{H}$.
Approach 2 is computationally efficient and well-suited for rapid evaluations, while Approach 1 remains more robust for pathologically distributed nodes or extreme derivative variations.

\begin{remark}
For high-degree polynomial fitting ($n_{\text{poly}} > 30$), standard Modified Gram-Schmidt suffers from loss of orthogonality. A second pass of orthogonalization ensures that the basis matrix $Q$ maintains a condition number of exactly one. When solving the overdetermined stacked system in Approach 2, QR factorization should be used exclusively \cite{higham2002accuracy,trefethen2020}.
For high-order polynomial approximation, Stage I employs a Householder-stabilized orthogonalization \cite{Gander2014ScientificComputing,dax2004householder} rather than the traditional discrete maxvol permutation. While a standard permutation matrix $P$ restricts the local coordinate map to a selection of standard basis vectors. The procedure circumvents dynamic coordinate switching and preserves the underlying analytical smoothness.
\end{remark}

\section{Convergence and Error Analysis}
\label{sec5}
This section provides a theoretical analysis for the MV-(C)V+Amethod. We will demonstrate how the combination of geometric stabilization (Stage I) and algebraic interpolation (Stage II) achieves near-machine-precision subspace interpolation. The total error is splitted into three components: reconstruction mapping error, coordinate propagation error, and algebraic approximation error, and provide explicit bounds for each component.
 
We first establish the properties of the MV coordinate reconstruction formula, proving that it preserves orthonormality and smoothness while introducing no geometric bias. For completeness, we recall the necessary notation and provide formal definitions.

\begin{definition}
Let $\widetilde U \in \mathbb{R}^{n \times p}$ be a MaxVol optimized Stiefel matrix partitioned as
$$
\widetilde U = \begin{bmatrix} \widetilde U_1 \\ \widetilde U_2 \end{bmatrix}, \quad \widetilde U_1 \in \mathbb{R}^{p \times p}, \quad \widetilde U_2 \in \mathbb{R}^{(n-p) \times p}.
$$
The MV coordinates are defined by the mapping $\Phi: \text{St}(n,p) \to \mathbb{R}^{(n-p) \times p}$:
$$
\Xi = \widetilde U_2 \widetilde U_1^{-1}.
$$
\end{definition}

\begin{lemma}[Reconstruction and Orthogonality Preservation] \label{Lemma1}
For any MV coordinates $\Xi \in \mathbb{R}^{(n-p) \times p}$, the reconstructed matrix
$$
\widetilde U(\Xi) = \begin{bmatrix} I_p \\ \Xi \end{bmatrix} (I_p + \Xi^T \Xi)^{-1/2}
$$
satisfies $\widetilde U(\Xi)^T \widetilde U(\Xi) = I_p$, and its MV coordinates are exactly $\Xi$.
\end{lemma}

\begin{proof}
Let $S = (I_p + \Xi^T \Xi)^{-1/2}$. Note that $S$ is symmetric. The matrix $\widetilde U$ can be written as
$$
\widetilde U = \begin{bmatrix} S \\ \Xi S \end{bmatrix}.
$$
Computing the Gram matrix:
$$
\widetilde U^T \widetilde U = \begin{bmatrix} S^T & S^T \Xi^T \end{bmatrix} \begin{bmatrix} S \\ \Xi S \end{bmatrix} = S^T S + S^T \Xi^T \Xi S = S(I_p + \Xi^T \Xi)S = S S^{-2} S = I_p.
$$
To verify the MV coordinates, we compute $\widetilde U_2 \widetilde U_1^{-1} = (\Xi S)S^{-1} = \Xi$.
\end{proof}

\begin{lemma} \label{lem:smooth_iso}
The mapping $\Xi \mapsto \widetilde{U}(\Xi)$ defined in Lemma 1 is $C^\infty$ on its domain. Furthermore, at $\Xi = 0$, the differential $d\widetilde{U}|_0(\Delta) = \begin{bmatrix} 0 \\ \Delta \end{bmatrix}$, which constitutes an isometry between the tangent space $T_0 \mathbb{R}^{(n-p) \times p}$ and the tangent space $T_{\widetilde{U}(0)} \text{St}(n,p)$, where the base point is $\widetilde{U}(0) = \begin{bmatrix} I_p \\ 0 \end{bmatrix}$.
\end{lemma}

\begin{proof}
The mapping is composed of the matrix inverse square root $S(\Xi) = (I_p + \Xi^T \Xi)^{-1/2}$. Since $I_p + \Xi^T \Xi$ is symmetric positive definite, the matrix resides in the open set of SPD matrices. The matrix function $A \mapsto A^{-1/2}$ is known to be $C^\infty$ on the SPD cone, hence $\widetilde{U}(\Xi)$ is $C^\infty$ on its domain.

  At the origin $\Xi=0$, we have $S(0) = I_p$. To find the differential $dS|_0$, we differentiate the identity $S(\Xi)^{-2} = I_p + \Xi^T \Xi$ at $\Xi=0$:
\begin{equation}
d(S^{-2})|_0 = -S^{-1}(dS)S^{-2} - S^{-2}(dS)S^{-1} \Big|_0 = -2 dS|_0.
\end{equation}
Meanwhile, the derivative of the right-hand side is $d(I_p + \Xi^T \Xi)|_0 = \Delta^T \Xi + \Xi^T \Delta \big|_{\Xi=0} = 0$. Consequently, $-2 dS|_0 = 0$, which implies $dS|_0 = 0$. 
By applying the product rule to the reconstruction formula   $\widetilde{U}(\Xi) = \begin{bmatrix} I_p \\ \Xi \end{bmatrix} S(\Xi)$ (see Lemma~\ref{Lemma1}), the differential in direction $\Delta$ is
\begin{equation} \label{eq:diff_at_zero}
d\widetilde{U}|_0(\Delta) = \begin{bmatrix} 0 \\ \Delta \end{bmatrix} S(0) + \begin{bmatrix} I_p \\ 0 \end{bmatrix} dS|_0 = \begin{bmatrix} 0 \\ \Delta \end{bmatrix}.
\end{equation}

To prove that $d\widetilde{U}|_0$ is an isometry, we verify that it preserves the inner product. The tangent space $T_0\mathbb{R}^{(n-p)\times p}$ is endowed with the standard Euclidean inner product $\langle \Delta_1, \Delta_2 \rangle = \operatorname{tr}(\Delta_1^T \Delta_2)$. For the Stiefel manifold $\text{St}(n, p)$, we equip it with the standard Euclidean metric inherited from the embedding space $\mathbb{R}^{n \times p}$, defined as $\langle \delta U_1, \delta U_2 \rangle = \operatorname{tr}(\delta U_1^T \delta U_2)$. From \eqref{eq:diff_at_zero} we know that
\begin{equation}
\langle d\widetilde{U}|_0(\Delta_1), d\widetilde{U}|_0(\Delta_2) \rangle = \operatorname{tr}\left( \begin{bmatrix} 0 & \Delta_1^T \end{bmatrix} \begin{bmatrix} 0 \\ \Delta_2 \end{bmatrix} \right) = \operatorname{tr}(\Delta_1^T \Delta_2) = \langle \Delta_1, \Delta_2 \rangle.
\end{equation}
Thus, the mapping is an isometry at the origin.
\end{proof}

\subsection{Error Propagation in Coordinate Systems}

\begin{definition} 
The geometric condition number of the MV-representation is $\kappa_{\text{geo}} = \|\widetilde U_1^{-1}\|_2$. This quantity measures the metric distortion between the Grassmann manifold and the Euclidean coordinate domain.
\end{definition}

\begin{theorem} \label{thm:err_prop}
Let $\Xi$ be the true MV coordinates and $\hat{\Xi} = \Xi + \Delta$ be the interpolated coordinates. Under the Householder-stabilized system, the reconstruction error for the orthogonal projector $P = UU^T$ satisfies
\begin{equation} \label{eq:err_bound_final}
\|\hat{P} - P\|_F \leq \sqrt{2} \cdot \|\widetilde U_1\|_2 \cdot \|\Delta\|_F + \mathcal{O}(\|\Xi\|_2 \|\Delta\|_F + \|\Delta\|_F^2),
\end{equation}
where $\|\widetilde U_1\|_2$ is the spectral norm of the upper $p \times p$ block of the reconstructed Stiefel representative.
\end{theorem}

\begin{proof}
Let $\Psi(\Xi) = \widetilde{U}(\Xi)\widetilde{U}(\Xi)^T$ be the mapping from the coordinate domain to the Grassmannian. Define the line segment $\gamma(t) = \Xi + t\Delta$ for $t \in [0,1]$. By the Mean Value Theorem for matrix-valued functions, we have
 $$
\hat{P} - P = \Psi(\Xi + \Delta) - \Psi(\Xi) = \int_0^1 d\Psi(\gamma(t))[\Delta] \, dt.
 $$
Taking the Frobenius norm on both sides yields
\begin{equation}
\|\hat{P} - P\|_F \leq \int_0^1 \left\| d\Psi(\gamma(t))[\Delta] \right\|_F dt \leq \sup_{t \in [0,1]} \left\| d\Psi(\gamma(t))[\Delta] \right\|_F. \label{eq:norm_integral}
\end{equation}
To bound the differential norm at any point $X \in \{\gamma(t)\}$, we evaluate $d\Psi(X)[\Delta] = (d\widetilde U)\widetilde U^T + \widetilde U(d\widetilde U)^T$. Using the reconstruction formula $\widetilde U(X) = \begin{bmatrix} I_p \\ X \end{bmatrix} S(X)$ where $S(X) = (I_p + X^T X)^{-1/2}$, the differential is
\begin{equation} \label{eq:du_exact_rigorous}
d\widetilde U = D_1 + D_2.
\end{equation}
where  $D_1 = \begin{bmatrix} 0 \\ \Delta \end{bmatrix} S$ and $D_2 = \begin{bmatrix} I_p \\ X \end{bmatrix} dS$.
Differentiating both sides of $S^{-2} = I_p + X^T X$ yields:
$$ d(S^{-2}) = \Delta^T X + X^T \Delta. $$
Applying the product rule to $S^{-2} = S^{-1}S^{-1}$ and using the matrix differential identity $d(S^{-1}) = -S^{-1}(dS)S^{-1}$, we expand the left side as
$ d(S^{-2})  = -S^{-1}(dS)S^{-2} - S^{-2}(dS)S^{-1}. $
Equating the two expressions and pre- and post-multiplying by $S$ give the Sylvester equation for $dS$:
$$ (dS)S^{-1} + S^{-1}(dS) = -S(\Delta^T X + X^T \Delta)S. $$
Its unique solution is given by the integral representation
$$ dS = \int_0^\infty e^{-S^{-1}\tau} [S(\Delta^T X + X^T \Delta)S] e^{-S^{-1}\tau} d\tau. $$
Taking the spectral norm and applying the bound $|e^{-S^{-1}\tau}|_2 \le e^{-\tau/|S|_2}$, the scalar integral can be evaluated to produce
$$ \|dS\|_2 \le \frac{\|S\|_2}{2} \|S(\Delta^T X + X^T \Delta)S\|_2 \le \|S\|_2^3 \|X\|_2 \|\Delta\|_F. $$
Substituting $d\tilde{U}$ into $d\Psi$, we decompose it into a principal part and a remainder
$$ d\Psi = (D_1 \tilde{U}^T + \tilde{U} D_1^T) + \underbrace{(D_2 \tilde{U}^T + \tilde{U} D_2^T)}_{\mathcal{O}(\|X\|_2\|\Delta\|_F)}. $$
Let $A = D_1 \tilde{U}^T = \begin{bmatrix} 0 \\ \Delta S \end{bmatrix} \tilde{U}^T$. Applying the triangle inequality to the remainder $\mathcal{R} = D_2 \tilde{U}^T + \tilde{U} D_2^T$ gives 
$$ \|\mathcal{R}\|_F \le 2 \|D_2\|_F \|\tilde{U}\|_2 \le 2 \left|\left| \begin{bmatrix} I_p \\ X \end{bmatrix} \right|\right|_2 \|dS\|_F \le 2\sqrt{p} \sqrt{1 + \|X\|_2^2} \cdot \|S\|_2^3 \|X\|_2 \|\Delta\|_F. $$
Thus the remainder is rigorously bounded by $\|\mathcal{R}\|_F = \mathcal{O}(\|X\|_2\|\Delta\|_F)$.  Next, we evaluate $\|A + A^T\|_F$. By the properties of the Frobenius norm,
\begin{equation}\label{eq:a_plus_at_norm}
\|A + A^T\|_F^2 = \|A\|_F^2 + \|A^T\|_F^2 + 2\operatorname{tr}(A^2) = 2\|A\|_F^2 + 2\operatorname{tr}(A^2).
\end{equation}
Since $\widetilde U$ has orthonormal columns ($\widetilde U^T \widetilde U = I_p$), we have $$\|A\|_F^2 = \operatorname{tr}(A^T A) = \operatorname{tr}(\widetilde U D_1^T D_1 \widetilde U^T) = \operatorname{tr}(D_1^T D_1) = \|D_1\|_F^2 = \|\Delta S\|_F^2.$$
Moreover, 
 $$
A^2 = \left( \begin{bmatrix} 0 \\ \Delta S \end{bmatrix} \widetilde U^T \right)^2 = \begin{bmatrix} 0 \\ \Delta S \end{bmatrix} \left( \begin{bmatrix} S & S X^T \end{bmatrix} \begin{bmatrix} 0 \\ \Delta S \end{bmatrix} \right) \widetilde U^T = \begin{bmatrix} 0 \\ \Delta S^2 X^T \Delta S \end{bmatrix} \widetilde U^T.
 $$
Thus, $\operatorname{tr}(A^2) = \operatorname{tr}(S X^T \Delta S^2 X^T \Delta S)$. By Cauchy-Schwarz inequality for trace inner products $|\operatorname{tr}(Y Z)| \leq \|Y\|_F \|Z\|_F$, letting $Y = S X^T \Delta S$ and $Z = S X^T \Delta S$, we get
 $$
|\operatorname{tr}(A^2)| \leq \|S X^T \Delta S\|_F^2 \leq \|S\|_2^4 \|X\|_2^2 \|\Delta\|_F^2.
 $$
Substituting back into \eqref{eq:a_plus_at_norm}, we obtain
 $$
\|A + A^T\|_F = \sqrt{2\|\Delta S\|_F^2 + \mathcal{O}(\|X\|_2^2 \|\Delta\|_F^2)} = \sqrt{2} \|\Delta S\|_F + \mathcal{O}(\|X\|_2  \|\Delta\|_F).
 $$
Using the sub-multiplicative property $\|\Delta S\|_F \leq \|S\|_2 \|\Delta\|_F$,
 $$
\|A + A^T\|_F \leq \sqrt{2} \|S\|_2 \|\Delta\|_F + \mathcal{O}(\|X\|_2  \|\Delta\|_F).
 $$
Combining the bounds for the principal part and the remainder $\mathcal{R}$, the differential norm at $X$ is bounded by
 $$
\|d\Psi(X)[\Delta]\|_F \leq \sqrt{2} \|S(X)\|_2 \|\Delta\|_F + \mathcal{O}(\|X\|_2   \|\Delta\|_F).
 $$
Returning to the supremum in \eqref{eq:norm_integral}, since $X \in \{\Xi + t\Delta\}$, we have $\|X\|_2 \leq \|\Xi\|_2 + \|\Delta\|_2$.
Furthermore, as established in Lemma~\ref{lem:smooth_iso}, the matrix function $S(\cdot)$ is smooth, ensuring $\|S(X)\|_2 = \|S(\Xi)\|_2 + \mathcal{O}(\|\Delta\|_F)$.
Since the upper block of $\widetilde U(\Xi)$ is $S(\Xi)$, we have $\|\widetilde U_1\|_2 = \|S(\Xi)\|_2$. Substituting these limits into the integral bound completes the strict derivation
$$
\|\hat{P} - P\|_F \leq \sqrt{2} \|\widetilde U_1\|_2 \|\Delta\|_F + \mathcal{O}(\|\Xi\|_2 \|\Delta\|_F + \|\Delta\|_F^2).
$$
The constant $\sqrt{2}$ arises from the geometry of the projector space, while the spectral norm $\|\widetilde U_1\|_2$ ensures that the reconstruction mapping does not amplify the coordinate-domain error.
\end{proof}

\subsection{Algebraic Stability of CV+A}

\begin{theorem}[Optimal Algebraic Conditioning of CV+A] \label{thm:arnoldi_cond}
Let $\mathbf{C} = \begin{bmatrix} D & O \\ I & D \end{bmatrix} \in \mathbb{R}^{2m \times 2m}$ be the augmented confluent operator, where $D = \operatorname{diag}(t_1,\ldots,t_m)$. The Arnoldi orthogonalization process applied to $\mathbf{C}$ with the starting vector $\mathbf{b} = [\mathbf{e}_m^T, \mathbf{0}_m^T]^T$ generates an augmented basis matrix $\mathcal{Q} = \begin{bmatrix} Q_f \\ Q_d \end{bmatrix} \in \mathbb{R}^{2m \times (k+1)}$ for Hermite interpolation. With appropriate reorthogonalization, the condition number of this matrix strictly satisfies $\kappa_2(\mathcal{Q}) = 1 + \mathcal{O}(\varepsilon_{\text{mach}})$.
\end{theorem}

\begin{proof}
The Confluent Vandermonde with Arnoldi (CV+A) approach circumvents the exponential ill-conditioning inherent to standard monomial bases by shifting the polynomial construction into an augmented Krylov subspace $\mathcal{K}_{k+1}(\mathbf{C}, \mathbf{b})$. 
Instead of generating a standard polynomial basis and subsequently differentiating it, the Arnoldi iteration is applied directly to the augmented operator $\mathbf{C}$. This process constructs polynomials that are orthonormal with respect to a discrete Sobolev-like inner product, which inherently couples the function evaluations and their derivatives. 

By mathematical construction, this process enforces the global orthonormality of the augmented block matrix, ensuring $\mathcal{Q}^T \mathcal{Q} = I_{k+1}$ in exact arithmetic. In finite-precision floating-point arithmetic, standard Arnoldi iterations suffer from a gradual loss of orthogonality. However, by employing Modified Gram-Schmidt with reorthogonalization (or Householder-based Arnoldi), the strict orthogonality is rigorously maintained. Consequently, the condition number of the basis can be maintained as $\kappa_2(\mathcal{Q}) = 1 + \mathcal{O}(\varepsilon_{\text{mach}})$\cite{paige2006mgs}. This guarantees that the linear system for extracting the polynomial coefficients, $\mathbf{A} = \mathcal{Q}^T \mathbf{Y}$ to be stable \cite{Zhang2024ArnoldiFitting}.
\end{proof}

\begin{lemma}[Forward Error of Arnoldi-Based Polynomial Evaluation] \label{lem:eval_error}
Let $p(t) = \sum_{j=0}^{n} c_j \phi_j(t)$ be a polynomial of degree $n$ expressed in the discrete orthonormal basis $\{\phi_j\}_{j=0}^n$ generated by the Arnoldi process. Suppose $p(t)$ is evaluated at any $t \in [-1, 1]$ using the recurrence relation encoded in the upper Hessenberg matrix $H \in \mathbb{R}^{(n+1) \times n}$ in floating-point arithmetic with machine epsilon $\varepsilon_{\text{mach}}$. The computed value, denoted as $\text{fl}(p(t))$, satisfies the forward error bound
\begin{equation}
|\text{fl}(p(t)) - p(t)| \leq C_{\text{eval}} \cdot n^{3/2} \cdot \varepsilon_{\text{mach}} \cdot \|\mathbf{c}\|_2,
\end{equation}
where $\mathbf{c} = [c_0, \dots, c_n]^T$ is the coefficient vector, and $C_{\text{eval}}$ is a moderate constant independent of $n$.
\end{lemma}

\begin{proof}
The evaluation of the basis functions $\phi_j(t)$ via the Arnoldi recurrence requires computing $w_{j+1} = (t w_j - \sum_{i=1}^j H_{i,j} w_i) / H_{j+1,j}$ at each step $j$. This is a generalization of the classic three-term recurrence to an $(n+1)$-term recurrence. 

According to standard floating-point error analysis for polynomial evaluation \cite{higham2002accuracy}, the local round-off error introduced at step $j$ when computing the inner product and scalar operations is bounded by $\mathcal{O}(j \varepsilon_{\text{mach}})$. Because the basis $\{\phi_j\}$ is constructed to be discrete orthonormal over the parameter nodes, the magnitudes of the basis functions and the elements of the Hessenberg matrix $H$ are strictly bounded, preventing the exponential amplification of these local errors.

To evaluate the full polynomial $p(t)$, the total accumulated error is the sum of the local errors weighted by the coefficients $c_j$. Applying the Cauchy-Schwarz inequality to the inner product of the coefficients and the accumulated basis errors yields
\begin{equation}
|\text{fl}(p(t)) - p(t)| \leq \sum_{j=0}^n |c_j| \cdot \mathcal{O}(j \varepsilon_{\text{mach}}) \leq \|\mathbf{c}\|_2 \left( \sum_{j=0}^n \mathcal{O}(j^2 \varepsilon_{\text{mach}}^2) \right)^{1/2} = \mathcal{O}(n^{3/2} \varepsilon_{\text{mach}}) \|\mathbf{c}\|_2.
\end{equation}
This confirms that the numerical evaluation of orthogonal polynomials via Arnoldi recurrence grows at most sub-quadratically with the degree $n$, which circumvents the exponential error growth associated with standard monomial basis evaluation \cite{brubeck2021vandermonde}.
\end{proof}

\begin{theorem}[Approximation Error of Global CV+A Interpolation] \label{thm:approx_order}
Let $f \in C^{2k+2}([0,1], \mathbb{R}^{(n-p) \times p})$ be the true coordinate function. Suppose $\hat{f}$ is the global Hermite interpolant of degree $2k+1$ constructed via the CV+A method over $m = k+1$ Chebyshev nodes of the first kind on $[0,1]$. Then the infinity-norm error satisfies
\begin{equation} \label{eq:alg_error}
\|f - \hat{f}\|_\infty \leq \frac{1}{2^{4k+2} (2k+2)!} \|f^{(2k+2)}\|_\infty + \mathcal{O}(k^{3/2} \varepsilon_{\text{mach}}),
\end{equation}
where the first term represents the exact analytical truncation error, and the second term accounts for the well-conditioned floating-point evaluation error, with $\varepsilon_{\text{mach}}$ being the machine epsilon.
\end{theorem}

\begin{proof}
To establish the error bound, we decompose the total interpolation error into the analytical truncation error in infinite precision and the numerical round-off error in finite-precision arithmetic. Let $p^*(t)$ be the exact theoretical Hermite interpolating polynomial of degree $2k+1$, and let $\hat{f}(t)$ be the numerical interpolant computed via the CV+A algorithm. By the triangle inequality,
\begin{equation} \label{eq:triangle_approx}
\|f - \hat{f}\|_\infty \leq \underbrace{\|f - p^*\|_\infty}_{\text{Analytical Error}} + \underbrace{\|p^* - \hat{f}\|_\infty}_{\text{Numerical Error}}.
\end{equation}

For the analytical error, applying the classical Cauchy remainder theorem for Hermite interpolation component-wise to $f(t)$ and taking the infinity norm on both sides gives the strict upper bound
$$ \|f(t) - p^*(t)\|_\infty \leq \frac{\max_{\xi \in [0,1]} \|f^{(2k+2)}(\xi)\|_\infty}{(2k+2)!} \prod_{i=0}^k (t - t_i)^2. $$
For Chebyshev nodes of the first kind affinely mapped from $[-1,1]$ to the interval $[0,1]$, the monic nodal polynomial $\omega(t) = \prod_{i=0}^k (t - t_i)$ satisfies the classical minimax bound $\max_{t \in [0,1]} |\omega(t)| \leq \frac{1}{2^{2k+1}}$. Since Hermite interpolation matches both values and first derivatives, the nodal polynomial is squared, yielding the global analytical bound
\begin{equation} \label{eq:analytical_bound}
\|f - p^*\|_\infty \leq \frac{1}{2^{4k+2} (2k+2)!} \|f^{(2k+2)}\|_\infty.
\end{equation}

For the numerical error, we must distinguish between the exact mathematical solution and its floating-point approximation. Let $\mathbf{A}^* = [A^*_0, A^*_1, \dots, A^*_{2k+1}]^T$ denote the exact coefficient vector that solves the augmented system $\mathcal{Q} \mathbf{A}^* = \mathbf{Y}$ in infinite precision. This vector corresponds exactly to the theoretical interpolant $p^*(t) = \sum_{j=0}^{2k+1} A^*_j \phi_j(t)$. In finite-precision arithmetic, solving the system yields a computed coefficient vector, denoted as $\hat{\mathbf{A}} = [\hat{A}_0, \hat{A}_1, \dots, \hat{A}_{2k+1}]^T$. The computed coefficients $\hat{\mathbf{A}}$ suffer from forward errors bounded by the condition number. As established in Theorem \ref{thm:arnoldi_cond}, Arnoldi orthogonalization guarantees $\kappa_2(\mathcal{Q}) = 1 + \mathcal{O}(\varepsilon_{\text{mach}})$. Consequently, the coefficient perturbation is strictly clamped near the precision floor: $\|\hat{\mathbf{A}} - \mathbf{A}^*\|_2 = \mathcal{O}(\varepsilon_{\text{mach}})$.

Finally, evaluating the numerical interpolant $\hat{f}(t) = \sum_{j=0}^{2k+1} \hat{A}_j \phi_j(t)$ using the Arnoldi recurrence introduces an accumulation of floating-point errors. By using Lemma \ref{lem:eval_error} for the global Hermite interpolant of degree $n = 2k+1$, the cumulative numerical evaluation error is rigorously bounded by $\mathcal{O}(k^{3/2} \varepsilon_{\text{mach}})$. This sub-quadratic growth elegantly circumvents the exponential error amplification associated with standard monomial bases. Summing the exact analytical truncation bound and this numerical evaluation bound completes the proof.

\end{proof}

\subsection{Total Error Bound}

\begin{theorem} \label{thm:total_error}
Let $\mathcal{U}(t)$ be a smooth subspace trajectory on the Grassmann manifold $\Gr(n,p)$ for $t \in [0,1]$, and let its MV-coordinate representation $\Xi(t) \in \mathbb{R}^{(n-p) \times p}$ be $C^{2k+2}$ continuous. If $\hat{P}(t)$ is the orthogonal projector reconstructed from the global degree $2k+1$ CV+A Hermite interpolant over $k+1$ Chebyshev nodes, the total subspace error satisfies
\begin{equation} \label{eq:total_bound_final}
\|\hat{P}(t) - P(t)\|_F \leq C \cdot \left( \frac{1}{2^{4k+2}(2k+2)!} \|\Xi^{(2k+2)}\|_\infty + \mathcal{O}(k^{3/2} \varepsilon_{\mathrm{mach}}) \right),
\end{equation}
where $P(t) = \mathcal{U}(t)\mathcal{U}(t)^T$ is the exact projector, and $C$ is a geometric projection constant. The stability of this bound relies on the geometric condition number $\kappa_{\mathrm{geo}} = \|\widetilde{U}_1^{-1}\|_2$ established in Stage I, which prevents the artificial inflation of the coordinate derivatives $\|\Xi^{(2k+2)}\|_\infty$.
\end{theorem}

\begin{proof}
The total approximation error on the Grassmann manifold is governed by the interplay between the algebraic interpolation accuracy in the Euclidean domain and the geometric distortion introduced by the coordinate mapping. Let $\Xi(t)$ be the exact MV coordinates of the subspace curve, and let $\hat{\Xi}(t)$ be the numerical interpolant computed via the global CV+A method. We denote the coordinate-domain error as $\Delta(t) = \hat{\Xi}(t) - \Xi(t)$.

From the geometric error propagation analysis in Theorem 1, the discrepancy between the exact projector $P(t) = \Psi(\Xi(t))$ and the numerically reconstructed projector $\hat{P}(t) = \Psi(\hat{\Xi}(t))$ is bounded by
\begin{equation} \label{eq:total_step1}
\|\hat{P}(t) - P(t)\|_F \leq \sqrt{2} \|\widetilde{U}_1\|_2 \|\Delta(t)\|_F + C_2 \|\Xi(t)\|_2 \|\Delta(t)\|_F + C_3 \|\Delta(t)\|_F^2,
\end{equation}
where $C_2$ and $C_3$ are bounded constants derived from the Taylor remainder of the retraction map over the compact parameter domain $t \in [0,1]$.

Since the reference point for the Stage I Householder stabilization is strategically chosen within the interpolation interval, and $\Xi(t)$ is smooth on the compact set $[0,1]$, the norm of the exact coordinate is globally bounded; that is, $\max_{t \in [0,1]} \|\Xi(t)\|_2 \leq M$ for some finite constant $M$.
As Theorem 3 establishes that the coordinate-domain error for the global degree $2k+1$ CV+A Hermite interpolant satisfies
\begin{equation} \label{eq:total_step2}
\|\Delta(t)\|_F \leq \mathcal{E}_{\mathrm{alg}}(k) := \frac{1}{2^{4k+2}(2k+2)!} \|\Xi^{(2k+2)}\|_\infty + C_\varepsilon k^{3/2} \varepsilon_{\mathrm{mach}},
\end{equation}
where $C_\varepsilon$ is a moderate numerical constant.

Substituting \eqref{eq:total_step2} into \eqref{eq:total_step1}, we get the composition of the geometric and algebraic errors. Because the polynomial degree $k$ drives spectral convergence, the analytical truncation term decays super-exponentially. Consequently, for sufficiently high degrees $k$, the coordinate error shrinks rapidly ($\|\Delta(t)\|_F \ll 1$), ensuring that the strictly quadratic term $C_3 \|\Delta(t)\|_F^2$ becomes entirely negligible compared to the linear terms.
Factoring out the dominant linear contribution and utilizing the intrinsic property $\|\widetilde{U}_1\|_2 \le 1$ from Theorem 1, gives
\begin{equation} \label{eq:linear_contribution}
\|\hat{P}(t) - P(t)\|_F \leq \left( \sqrt{2} + C_2 M \right) \mathcal{E}_{\mathrm{alg}}(k) + \mathcal{O}(\mathcal{E}_{\mathrm{alg}}(k)^2).
\end{equation}

Based on the above analysis, we can see that the Stage I geometric stabilization enforces $\kappa_{\mathrm{geo}} \approx \sqrt{p}$ and ensures the local coordinates remain centered (i.e., $M$ is tightly controlled). Thus, the structural distortion multiplier $(\sqrt{2} + C_2 M)$ is bounded and strictly independent of the polynomial degree $k$. Absorbing it into a unified geometric projection constant $C$, we obtain the spectral error bound
 $$
\|\hat{P}(t) - P(t)\|_F \leq C \cdot \left( \frac{1}{2^{4k+2}(2k+2)!} \|\Xi^{(2k+2)}\|_\infty + \mathcal{O}(k^{3/2} \varepsilon_{\mathrm{mach}}) \right).$$
\end{proof}
The above results indicate that the MV-(C)V+A method achieves the same convergence rate and stability as its Euclidean counterpart, with error limited only by machine precision.

\section{Numerical Examples}
 \label{sec6}
In this section, we created several examples to compare the proposed MV-(C)V+A approach with traditional approaches in numerical stability as polynomial degree increases; we also consider what interpolation accuracy is achieved on $\Gr(n,p)$ for both Lagrange and Hermite settings.

\subsection*{Example 1. High-Order Interpolation and Dual Stability Analysis}

We consider the Grassmann manifold $\Gr(1000,10)$ with a highly nonlinear curve generated by transcendental functions
\begin{equation} \label{eq:transcendental_curve}
Y(t) = Y_0 + \sin(3t)Y_1 + \cos(3t)Y_2 + \exp(t)Y_3, \quad t \in [0,1],
\end{equation}
where $Y_i \in \mathbb{R}^{1000\times10}$ are uniformly distributed. The true subspace trajectory is $U(t) = \mathrm{qr}(Y(t)) \in \Gr(1000,10)$. To avoid Runge's phenomenon, we select $m=8$ Chebyshev nodes:
\begin{equation}
t_i = \frac{1}{2} - \frac{1}{2}\cos\left(\frac{\pi i}{7}\right), \quad i = 0,\ldots,7,
\end{equation}
which gives Lagrange degree 7 and Hermite degree 15 interpolation. The fourth node $t_{\text{ref}} \approx 0.3887$ serves as the reference for local coordinates.

The geometric condition number $\kappa_{\text{geo}} = \|\widetilde U_1^{-1}\|_F$ quantifies how coordinate errors amplify upon manifold reconstruction. Theorem~\ref{thm:total_error} gives $\|\hat P - P\|_F \leq \sqrt{2} \|\widetilde U_1\|_2 \|\Delta\|_F$, with $\|\widetilde U_1\|_2 = 1/\kappa_{\text{geo}}$. Table~\ref{tab:geom_cond} compares $\kappa_{\text{geo}}$ across the eight nodes for different stabilization strategies.

\begin{table}[htbp]
\centering
\caption{Geometric condition number $\|\widetilde U_1^{-1}\|_F$ across eight Chebyshev nodes. The theoretical minimum for the Frobenius norm is $\sqrt{p} \approx 3.16$, corresponding to $\widetilde U_1 = I_p$.}
\label{tab:geom_cond}
\begin{tabular}{c|c|c|c}
\hline
\textbf{Point ($t$)} & \textbf{Before MaxVol} & \textbf{Original (Jensen)} & \textbf{ MV-(C)V+A (Ours)} \\
\hline
0.0000 & 141987.72 & 34.83 & \textbf{3.54} \\
0.0495 & 484.49 & 32.99 & \textbf{3.45} \\
0.1883 & 146.85 & 29.44 & \textbf{3.26} \\
{\bf 0.3887} & {\bf 146.52} & {\bf27.75} & \textbf{3.16} \\
0.6113 & 285.63 & 29.40 & \textbf{3.28} \\
0.8117 & 658.94 & 33.40 & \textbf{3.60} \\
0.9505 & 187.41 & 37.51 & \textbf{3.92} \\
1.0000 & 160.51 & 39.30 & \textbf{4.05} \\
\hline
\end{tabular}
\end{table}

Without geometric stabilization, the local coordinate suffers distortion, with $\kappa_{\text{geo}}$ spiking to $1.42\times10^5$ at the left boundary, amplifying coordinate errors by five orders of magnitude and rendering high-precision interpolation impossible.
The original row-permutation strategy bounds this distortion but still permits $\kappa_{\text{geo}}$ to fluctuate between 27 and 39. While a substantial improvement, this still amplifies errors by an order of magnitude.
As shown in Table~\ref{tab:geom_cond}, our Householder-based MV-(C)V+A curve achieves conditioning close to the theoretical optimum. At the reference point $t_{\text{ref}}=0.3887$, we attain exactly $\kappa_{\text{geo}} = \sqrt{p} \approx 3.16$, corresponding to $\widetilde U_1 = I_p$. Even at domain boundaries, MV-(C)V+A maintains $\kappa_{\text{geo}} \leq 4.05$, reducing the geometric error amplification described in Theorem~\ref{thm:total_error}. The near-optimal conditioning preserves coordinate-domain interpolation accuracy upon manifold reconstruction.

Figure~\ref{fig:dual_stability} compares Riemannian normal coordinates, standard local coordinates, original MV coordinates and the proposed MV-(C)V+A. The left panel shows reconstruction error $\|\hat P(t) - P(t)\|_F$; the right panel shows orthogonality loss $\|\hat U(t)^T \hat U(t) - I_p\|_F$, evaluated over 200 test points.

\begin{figure}[htbp]
\centering 
\includegraphics[width=1.1\textwidth]{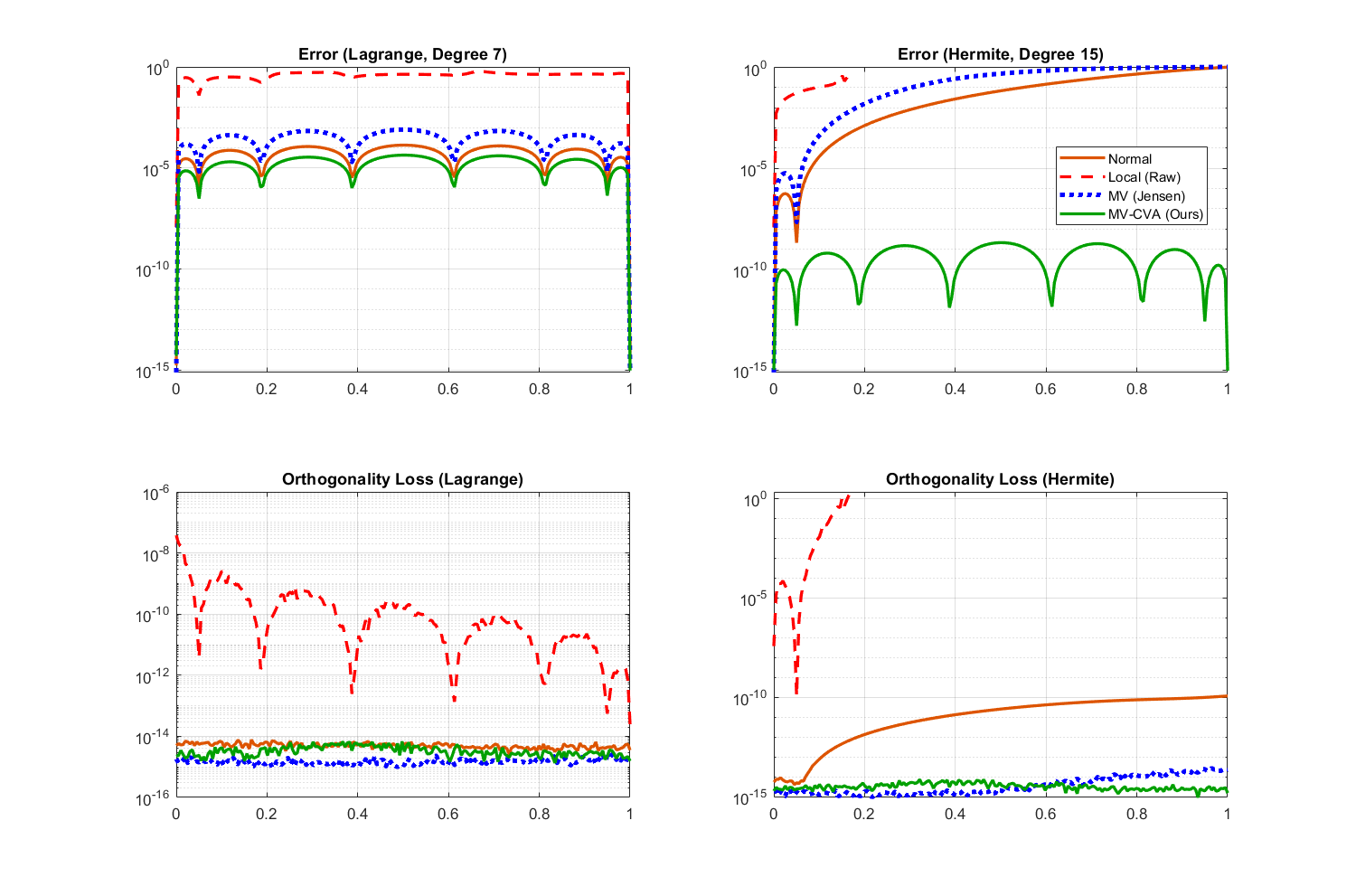}
\caption{Left: Relative reconstruction error $\|\hat P(t) - P(t)\|_F / \|P(t)\|_F$. Right: Orthogonality loss $\|\hat U(t)^T \hat U(t) - I_p\|_F$. }
\label{fig:dual_stability}
\end{figure} 
At degree $k=15$, the confluent Vandermonde matrices used by Normal, Local, and original MV coordinates become exponentially ill-conditioned ($\kappa > 10^{16}$) which causes the corresponding error curves to diverge, resulting in a severe loss of accuracy.
In contrast, MV-(C)V+A bypasses this algebraic barrier entirely. The Arnoldi process maintains perfect conditioning ($\kappa_2(\mathcal{Q}) = 1$) for the augmented basis matrix $\mathcal{Q} = [Q_f; Q_d]$. As shown in Figure~\ref{fig:dual_stability}, MV-(C)V+A's error decays smoothly to nearly machine precision. The right panel confirms that reconstructed matrices $\hat U(t)$ preserve orthonormality.

To test algorithmic limits, we increase nodes to $m=18$, which gives a degree $17$ Lagrange and degree $35$ Hermite interpolation. Figure~\ref{fig:dual_stability2} shows that at degree $35$, baseline methods relying on confluent Vandermonde matrices suffer complete algebraic collapse, with errors exploding toward $10^0$ and cascading orthogonality loss. MV-(C)V+A robustly maintains machine-precision accuracy in the domain.

\begin{figure}[H]
\centering 
\includegraphics[width=1.1\textwidth]{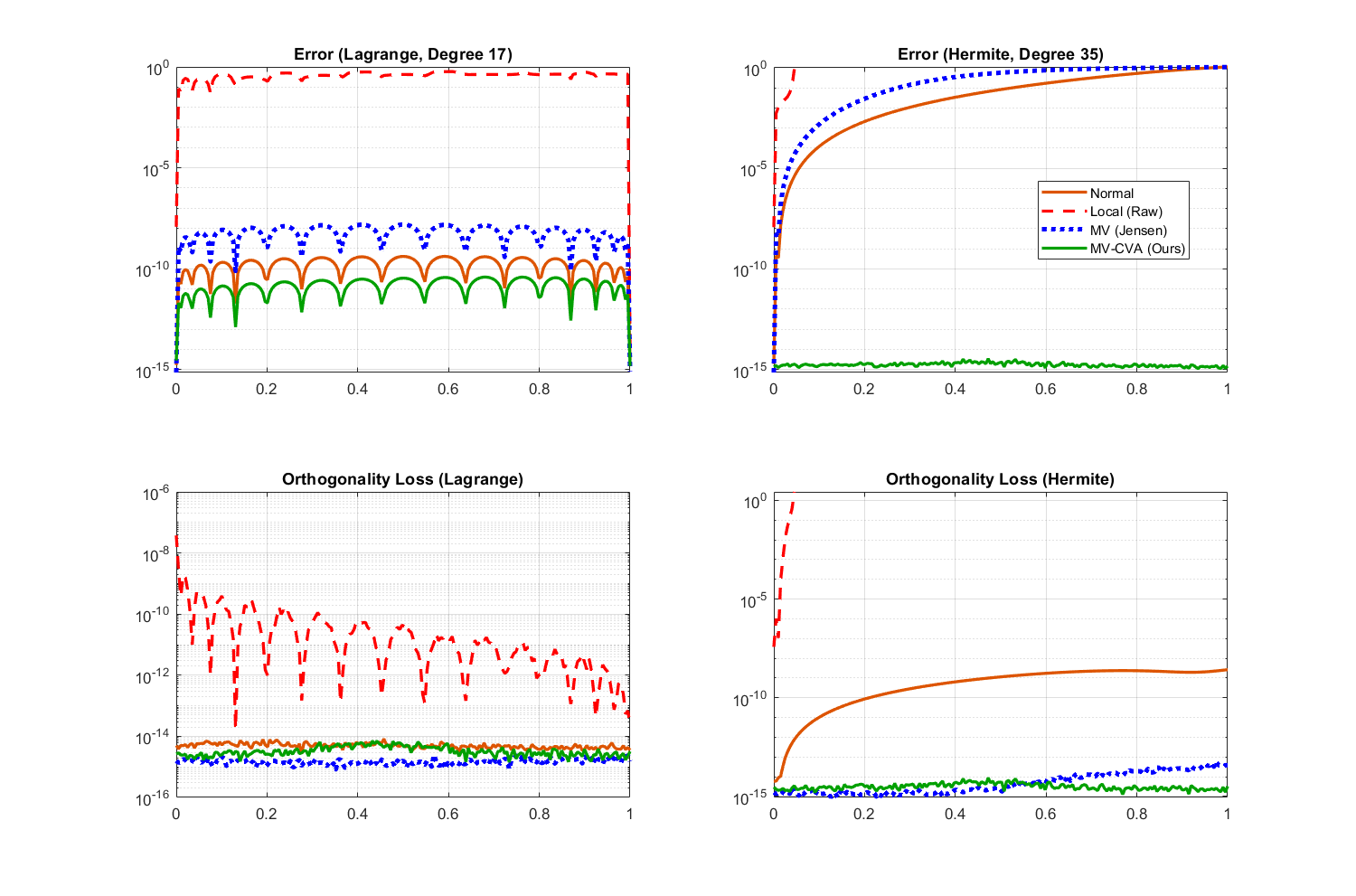}
\caption{Left: Relative reconstruction error $\|\hat P(t) - P(t)\|_F / \|P(t)\|_F$. Right: Orthogonality loss $\|\hat U(t)^T \hat U(t) - I_p\|_F$. }
\label{fig:dual_stability2}
\end{figure}

\subsection*{Example 2. Sensitivity to Data Perturbations and Noise Amplification}
In Example 1, we investigate the performance of the proposed method. In practical applications, e.g., Model Order Reduction (MOR), often involve manifold data containing external perturbations (e.g., SVD truncation errors or sensor noise). Therefore, this experiment explicitly validates the conditioning established in Theorem~\ref{thm:total_error} by analyzing how each algorithm amplifies inherent data noise. 

We adopt the same base construction on $\Gr(1000, 10)$ with the transcendental trajectory defined in \eqref{eq:transcendental_curve}. We select a moderate number of Chebyshev nodes, $m=10$, which corresponds to a degree 19 Hermite interpolation. At this degree, the analytical approximation error is theoretically well below $10^{-12}$. 

To simulate the influence of inexactness in practical data, we use a controlled, uniformly distributed artificial noise $\mathcal{E}_i \in \mathbb{R}^{n \times p}$ into the raw data matrices before orthogonalization:
 $$
    Y_{\text{noisy}}(t_i) = Y_{\text{exact}}(t_i) + \epsilon \cdot \frac{\mathcal{E}_i}{\|\mathcal{E}_i\|_F}, \quad i = 0, \ldots, m-1,
 $$
where the noise magnitude is strictly calibrated to $\epsilon = 10^{-10}$. We then extract the perturbed basis $U_{\text{noisy}}(t_i) = \mathrm{qr}(Y_{\text{noisy}}(t_i))$ and its exact analytical derivative to feed into the interpolation approaches. We track the reconstruction error against the true, unperturbed curve $P_{\text{true}}(t)$.

According to classical perturbation theory, the numerical output error is bounded by the sum of the theoretical approximation error and the data noise amplified by the condition number of the interpolation matrix system ($\kappa_{\text{alg}} \cdot \epsilon$). 

\begin{figure}[htbp]
\centering
 \includegraphics[width=1.05\textwidth]{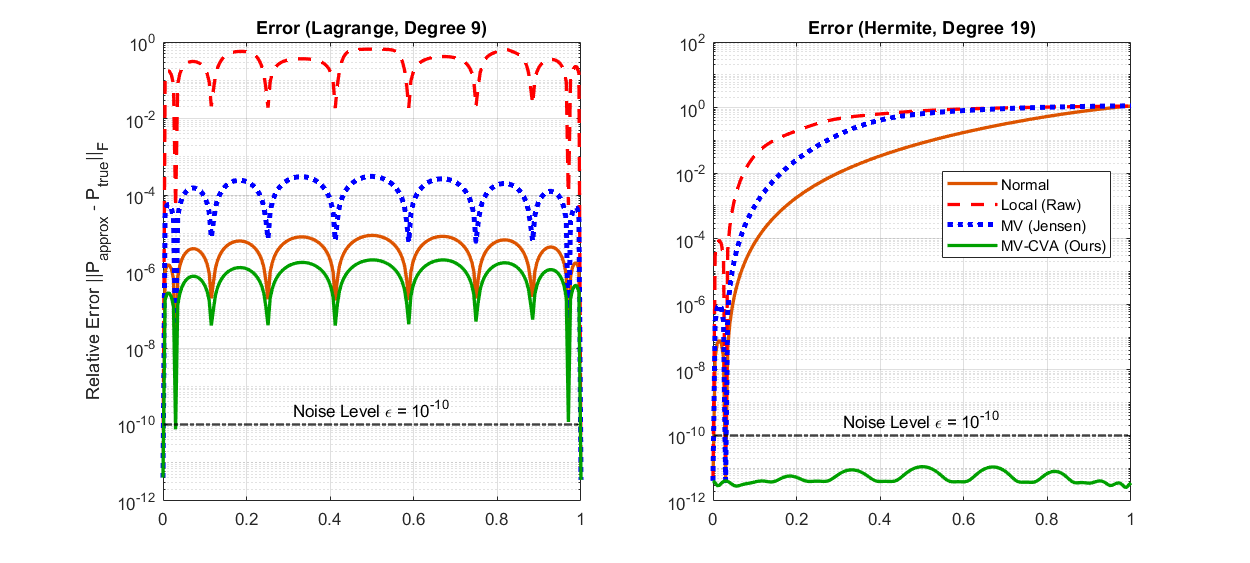}
\caption{Reconstruction error under data perturbation ($\epsilon = 10^{-10}$)}
\label{fig:noise_sensitivity}
\end{figure}

For the Normal, Local, and original MV coordinates, the underlying confluent Vandermonde matrix at degree 19 exhibits a condition number $\kappa_{\text{alg}} \approx 10^{13}$. Consequently, the microscopic data perturbation $\epsilon = 10^{-10}$ causes the total reconstruction error to be greatly magnified. As a comparison, the MV-(C)V+A  approach is much more stable, which strictly limits the error growth, mirroring the optimal conditioning $\kappa_{\text{alg}} = 1$. The total reconstruction error of MV-(C)V+A remains uniformly bounded near the injected noise level ($\approx 10^{-10}$). The numerical results show that MV-(C)V+A is robust for parametric subspace problems where data imperfections are unavoidable.

\subsection*{Example 3. Parametric Reduced-Order Modeling for 1D Helmholtz
Equation }

The Helmholtz equation is known to be difficult to parameterize globally due to physical resonances \cite{ernst2012difficult}, which cause the solution manifold to oscillate and twist highly nonlinearly with respect to the wavenumber. In this example, we consider the one-dimensional interior Helmholtz boundary value problem parameterized by the wavenumber $k$. The governing equation on the spatial domain $\Omega = [0, 1]$ is given by
\begin{align}
    - \frac{d^2 u(x; k)}{dx^2} - k^2 u(x; k) &= f(x), \quad x \in \Omega, \label{eq:1d_helmholtz_pde} \\
    u(0; k) = 0, \quad u(1; k) &= 0, \label{eq:1d_helmholtz_bc}
\end{align}
where $f(x)$ represents localized source excitations. 

Applying a standard spatial discretization (e.g., finite differences) on a uniform grid transforms the continuous problem into a parameter-dependent discrete linear system
\begin{equation} \label{eq:1d_helmholtz_discrete}
    A(k) Y(k) = F, \quad \text{where} \quad A(k) = L - k^2 I_n.
\end{equation}
Here, $L \in \mathbb{R}^{n \times n}$ is the discrete negative Laplacian matrix encoding the Dirichlet boundary conditions, $I_n$ is the identity matrix, and $F \in \mathbb{R}^{n \times p}$ contains $p$ distinct discrete source vectors. The corresponding full-order state responses are collected in $Y(k) \in \mathbb{R}^{n \times p}$. For this numerical experiment, the full system dimension is $n=500$, and the target POD subspace dimension is truncated at $p=8$.

Unlike empirical data fitting where derivative information is unavailable or requires unstable finite difference approximations, PDE-based ROM allows for the exact, inexpensive extraction of tangent vectors by the governing equations. By differentiating the state equation (\ref{eq:1d_helmholtz_discrete}) with respect to $k$, we obtain $A'(k)Y(k) + A(k)Y'(k) = 0$. Since $A'(k) = -2k I_n$, the exact analytical derivative of the state matrix is computed by solving one additional linear system with the already-factorized operator $A(k)$:
 $$
Y'(k) = A(k)^{-1} (2k Y(k)).
 $$
The matrices $Y(k)$ and their derivatives $Y'(k)$ are subsequently processed using the formulas in Section 3 to yield the orthogonal Stiefel representatives $U(k)$ and their tangent vectors $U'(k)$.
We are interested in the high-frequency parameter domain $k \in [10, 20]$, normalized to $t \in [0, 1]$. To construct a global, high-fidelity surrogate model, we select $m=12$ Chebyshev nodes. Extracting both $U(t_i)$ and $U'(t_i)$ at these nodes formulates a global Hermite interpolation problem of degree $2m-1 = 23$.

\begin{figure}[htbp]
\centering
  \includegraphics[width=0.95\textwidth]{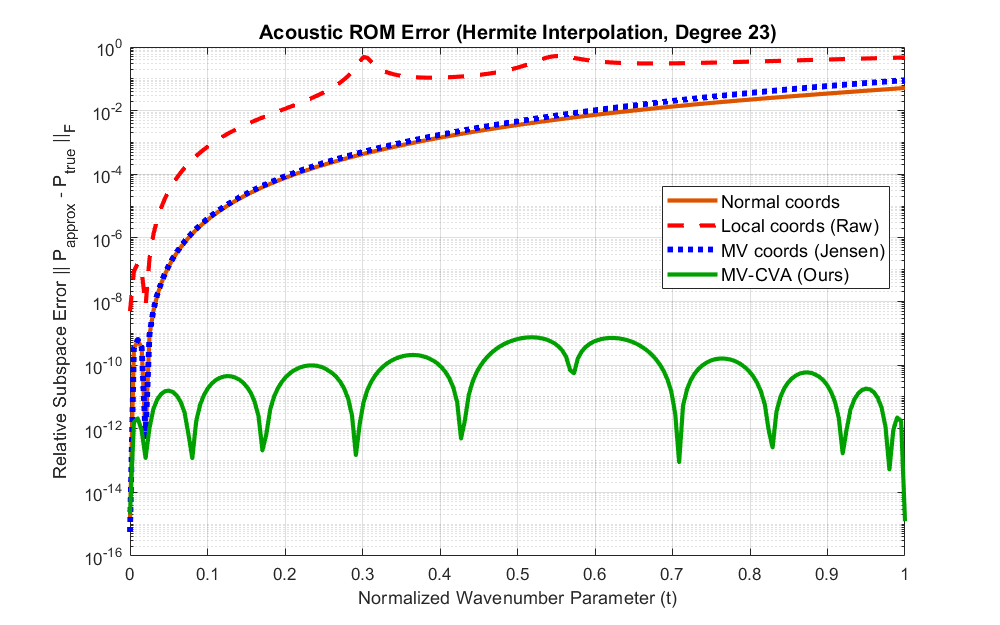}
\caption{Relative POD subspace error $\| \hat{P}(t) - P_{\text{true}}(t) \|_F / \| P_{\text{true}}(t) \|_F$ for the parametric Helmholtz ROM. The global Hermite interpolation is of degree 23. Traditional methods relying on confluent Vandermonde systems suffer complete algebraic collapse, whereas MV-(C)V+A  preserves machine precision.}
\label{fig:helmholtz_error}
\end{figure}

In Figure \ref{fig:helmholtz_error}, we compare the performance of several methods when derivative data is used at high degrees. For Hermite interpolation, the baseline methods must construct and solve confluent Vandermonde systems. At degree 23, the condition number of these matrices exceeds $10^{16}$ which causes numerical breakdown. Consequently, the baseline error degrades to $\mathcal{O}(10^0)$, rendering standard coordinate methods inapplicable. In contrast, the MV-(C)V+A approach avoids the instability entirely. Its Arnoldi-based orthogonalization ensures that the confluent interpolation system remains well conditioned ($\kappa_2 = 1$), regardless of polynomial degree. The MV-(C)V+A method successfully performs degree-23 global Hermite interpolation, maintaining relative subspace errors near $10^{-10}$ throughout the frequency range. The experiment confirms that MV-(C)V+A provides a stable foundation for constructing high-order, derivative-enhanced models in complex physical applications.

\section{Conclusion}
\label{sec7}

This paper introduced the MV-(C)V+A framework for high-order interpolation on the Grassmann manifold, which addresses the challenges of geometric distortion and algebraic ill-conditioning. By combining Householder-stabilized maximum volume local coordinates with Arnoldi-orthogonalized polynomial bases, the approach achieves a high level of numerical stability.

From a geometric perspective, the framework controls the local condition number near its theoretical lower bound ($\kappa_{\text{geo}} \approx \sqrt{p}$) \cite{goreinov2010how}, which mitigates the amplification of coordinate-induced errors. From an algebraic standpoint, the CV+A formulation yields well-conditioned interpolation systems with $\kappa_2 = 1$. The coordinate derivative mapping is expressed in a purely algebraic form, avoiding the numerical instability associated with differentiating SVD-based Riemannian mappings. Numerical experiments, including parametric reduced-order modeling of the Helmholtz equation, show that the MV-(C)V+A framework maintains near-machine-precision accuracy in high-degree regimes where classical approaches fail.

Future work will focus on adaptive strategies for polynomial degree and node selection, and extensions to other matrix manifolds such as the Stiefel manifolds and the rotation groups \cite{zimmermann2017matrix, zimmermann2022computing, massart2020quotient}. Beyond reduced-order modeling for physical PDEs, stable high-order Grassmann interpolation holds significant potential for constructing parameter-dependent preconditioners and coarse spaces in two-level stationary iterative methods \cite{ciaramella2024gentle}. Extending the approach to multivariate parameter domains via multivariate Arnoldi-type constructions \cite{zhang2025multivariate, zhu2025convergence} is a promising direction.

\bibliographystyle{plain}
\bibliography{reference}

\end{document}